\documentclass[reqno,11pt]{amsart}
\usepackage{amsmath}
\usepackage{amssymb}
\usepackage{amsfonts}
\usepackage{color}

\allowdisplaybreaks
\newtheorem{theorem}{Theorem}[section]

\newtheorem{lemma}[theorem]{Lemma}

\theoremstyle{definition}

\theoremstyle{remark}
\newtheorem{remark}[theorem]{Remark}
\theoremstyle{example}

\theoremstyle{problem}

\theoremstyle{fact}

\numberwithin{equation}{section}

\begin{document}
\title[inverse transmission eigenvalue problem]{On the inverse transmission eigenvalue problem with a piecewise $W_2^1$ refractive index}

\author[T. Liu]{Tao Liu}
\thanks{}
\address{School of Mathematics and Statistics, Shaanxi Normal
University, Xi'an 710062, PR China; School of Mathematics, Hangzhou Normal University, Hangzhou 311121, PR China}
\email{liutaomath@163.com}
\author[K. Lyu]{Kang Lyu*}\thanks{* Corresponding author}
\address{School of Mathematics and Statistics, Nanjing
	University of Science and Technology, Nanjing 210094, PR China}
\email{lvkang201905@outlook.com}
\author[G. Wei]{Guangsheng Wei}
\address{School of Mathematics and Statistics, Shaanxi Normal
University, Xi'an 710062, PR China} \email{weimath@vip.sina.com}
\author[C.-F. Yang]{Chuan-Fu Yang}
\address{ School of Mathematics and Statistics, Nanjing University of Science and Technology, Nanjing 210094, PR China}
\email{chuanfuyang@njust.edu.cn}


\begin{abstract}
In this paper,  we consider the inverse spectral problem of determining the spherically symmetric refractive index in a
bounded spherical region of radius $b$. Instead of the usual case of the refractive index $\rho\in W^2_2$,
by using singular Sturm-Liouville theory, we {first} discuss the case when the refractive index $\rho$ is a piecewise $ W^1_2$ function. We prove that if $\int_0^b \sqrt{\rho(r)} dr<b$, then $\rho$ is uniquely determined by all special transmission eigenvalues; if $\int_0^b \sqrt{\rho(r)} dr=b$, then all special transmission eigenvalues with some additional information can uniquely determine $\rho$. We also consider the mixed spectral problem and obtain that $\rho$ is uniquely determined from partial information of $\rho$ and the ``almost real subspectrum".
\end{abstract}

\maketitle

\section{Introduction}
The interior transmission  problem appears in scattering theory for inhomogeneous acoustic and electromagnetic media, which was introduced by Kirsch, Colton and Monk \cite{colmon,  kir}.
 It is a non-selfadjoint problem for two fields $w$ and $v$:
	\begin{align}\label{transmission}
\begin{cases}
\Delta w+\lambda \rho(\mathbf{x})w=0,  &  \mathbf{x}\in \Omega, \\
\Delta v+\lambda v=0,& \mathbf{x}\in \Omega,\\
   w=v, \frac{\partial w}{\partial \mathbf{v}}=\frac{\partial v}{\partial \mathbf{v}},& \mathbf{x}\in \partial \Omega.
\end{cases}
\end{align}
Here $\lambda$ is the spectral parameter,   $\Omega$ is a bounded and simply connected set in  $\mathbb{R}^n$ with smooth boundary, $\mathbf{v}$ is the outward unit normal to $\partial \Omega$,  $\rho(\mathbf{x})$ denotes the
refractive index of the medium \cite{colcoy, col1}.
	 The complex values of $\lambda$ for which a nontrivial
solution $(w, v)$ exists are called transmission eigenvalues.
 We refer to \cite{ashann,  ashsima,cak1,  cak3,cak} for the results about the existence  of transmission eigenvalues and bound of eigenvalues of Laplacian, the relation of the bound of transmission eigenvalues and Laplacian eigenvalues.

An interesting and important issue related to the interior transmission problem is the  corresponding inverse spectral problem. Namely, whether we can uniquely  determine $\rho$ in $\Omega$ if all or the certain subset of the transmission eigenvalue are known.
If   $n = 3$, $\Omega=\Omega_b$ is a ball of radius $b>0$ centered at the origin and $\rho(\mathbf{x})$ is
spherically symmetric ($\rho(\mathbf{x}) = \rho(r), r = |\mathbf{x}|$), then problem \eqref{transmission} can be transformed into the  one-dimensional eigenvalue problem \cite{cak2, mcl}.   In this paper, we consider  inverse
 problems of recovering $\rho(r)$ from transmission eigenvalues with spherically symmetric eigenfunctions $(\omega,v)$. Then the inverse problem is equivalent to recovering $\rho$ from  eigenvalues of the special transmission eigenvalue problem $Q(\rho)$
\begin{align}\label{string}
\left\{\begin{aligned}
&-u^{\prime \prime }=\lambda \rho(r)u,\qquad 0<r<b,   \\
&u(0)=0=u'(b)\frac{\sin( \sqrt{\lambda}b)}{\sqrt{\lambda}}-u(b){\cos (\sqrt{\lambda}b)}.
\end{aligned}\right.
\end{align}

 Inverse spectral problem for $Q(\rho)$ was  first studied by ~McLaughlin-Polyakov~\cite{mcl} and then  considered by a number of authors {\cite{but, but2,col,gin}}.  Previous literature on inverse spectral analysis for problem $Q(\rho)$ always considered that $\rho$ is a  $W_2^2$ or piecewise $C^2$  function \cite{act,gin,xux, yang1}.
 Aktosun-Gintides-Papanicolaou \cite{act}
 showed that if $a<b$, then  all special transmission eigenvalues uniquely determine $\rho$; if $a=b$,  then all special transmission eigenvalues together with the  additional constant $\gamma$ uniquely determine $\rho$. Here
 \begin{align}\label{definitiona}
 a=\int_0^b \sqrt{\rho(r)} dr.
 \end{align}
{In particular, when $a = b$,
 for the unique determination of the potential, the additional constant $\gamma$ is necessary to be known \cite{but1}}.
 Wei-Xu \cite{wei} considered the case when $a>b$. They showed that $\rho$ is uniquely determined by all special transmission eigenvalues and normalizing constants corresponding to the partial simple eigenvalues. Later, Yang-Buterin \cite{yang1} gave the uniqueness theorem from the data involving fractions of the special transmission eigenvalues. 
{See \cite{ashcmp,chu,fre,ges,liu1,liucmp,liucpam, mar,xu} more results about  eigenvalue problems and related inverse problems.}

 In this paper, we first investigate inverse spectral problems related to \eqref{string} for a piecewise $W_2^1$ refractive index.
   In this case,  problem \eqref{string} models a complicated medium for a less smooth refractive index with several layers where the index has jumps between each layer.  This uniqueness question shows that less smooth materials with layers can be determined  from the scattered far fields \cite{mcl}.
   It can also be used for the numerical investigation of inverse problems using a piecewise constant approximation of the refractive index \cite{gin1}.
 We consider that $\rho$ has a jump discontinuous point $b_1$. Namely, $\rho\in W_2^1\left((0, b_1)\cup (b_1, b)\right)$ and satisfies
\begin{align}
\rho(b_1+)=b_2\rho(b_1-),\ b_2>0\ \mathrm{and}\ b_2\neq 1,\ \rho(r)>0\ \mathrm{for\ any\ } r\in [0,b_1)\cup (b_1,b].   \label{piecewiseac}
\end{align}
Note that $b_2\neq 1$ is a natural assumption since $\rho$ is continuous {on} the whole interval if $b_2=1$.

The relaxation of the refractive index makes inverse transmission problems much more complicated.
If $\rho\in W_2^2(0,b)$, by Liouville transformation \eqref{liouville}, we can transform the equation $-u''=\lambda \rho u$ into a Sturm-Liouville (SL) equation with the potential
$q\in L^2(0,a)$.  The mapping
\begin{align}\nonumber
\mathcal{M}: \rho\rightarrow (\rho(b),\rho'(b),q(x)  )
\end{align}
 is injective.
 If $\rho$ is a piecewise $W^1_2$ function, the reduction to the potential form is possible, but the potential is a distribution from $W_2^{-1}(0,a)$ and eigenfunctions have a discontinuous point $d$. The  mapping $\mathcal{M}$ cannot be extended directly to the case when $\rho$ is a piecewise $W^1_2$ function. In this case, we can reformulate
the information of $(\rho'(b), q(x))$  as  $\sigma(x)$, where  $\sigma\in L^2(0,a)$ is the anti-derivative function of $q$. The mapping from $\rho$ to $(\rho(b),  \sigma(x),d, d_1)$ is injective (see Section 2 for  definitions of $d, d_1$). {By Liouville transformation,  the Barcilon formula refined by Shen \cite{shen} with $\rho$ to be a piecewise {continuously differentiable} function is a consequence of certain asymptotic formula.}
In this paper, by studying the discontinuous SL problem with  singular potentials, namely,
recovering $(\sigma,d,d_1)$  instead of $q$ in classical SL theory, we  obtain  uniqueness theorems even dropping the information of  $\rho'(b)$ (see Theorem \ref{theorem1}).
See \cite{alb, bondarenko,hr,hry} for some results on the singular SL problem without discontinuities.

The structure of this paper is as follows. In Section 2, we use ~Liouville transformation to transform \eqref{string} into the  discontinuous SL problem with  singular potentials. Section 3 gives the integral representation of the initial solution. Section 4 introduces the Weyl-Titchmarsh function of the discontinuous~SL problem with singular potentials and proves the corresponding uniqueness theorem. We also give the high-energy asymptotic behavior of the Weyl-Titchmarsh function. In Section 5, we use the  discontinuous  SL problem with  singular potentials to study inverse transmission eigenvalue problems by all eigenvalues. Section 6 studies properties of ``almost real subspectrum'' $\{\mu_m\}_{m=1}^\infty$, which are real except for finite many eigenvalues. We  show that the ``almost real subspectrum'' $\{\mu_m\}_{m=1}^\infty$ and some information on the refractive index uniquely determine $\rho$.

\section{Liouville transformation}

Let $u(r,\lambda)$ be the solution of $-u''=\lambda \rho u$ satisfying the initial condition $u(0,\lambda)=0, u'(0,\lambda)=1.$  It is known \cite%
{act} that special transmission eigenvalues of \eqref{string} coincide with the
zeros of its characteristic function
\begin{equation}
D (\lambda):=\left\vert
\begin{array}{cc}
\frac{\sin (\sqrt{\lambda}b )}{\sqrt{\lambda} } & u(b,\lambda) \\
\cos (\sqrt{\lambda}b ) & u^{\prime }(b,{\lambda})
\end{array}%
\right\vert .  \label{eq1.4}
\end{equation}%
Let $\{\lambda _{k}\}_{k=1}^{\infty }$ denote the eigenvalues of \eqref{string}
with account of multiplicity. Then according to Hadamard's factorization
theorem, we have
\begin{equation}
D (\lambda)=\gamma \lambda ^{s}\prod_{\lambda _{k}\neq 0}\left(1-\frac{\lambda }{\lambda
	_{k}}\right),  \label{eq1.5}
\end{equation}%
where   $s\geq 1$ is the multiplicity of the eigenvalue $\lambda =0$, $
\gamma \in \mathbb{R}$.

By Liouville transformation
\begin{align}\label{liouville}
x=\int_0^r\sqrt{\rho(t) }dt,
\end{align}
we can transform $-u''=\lambda \rho u$ into  the  discontinuous SL problem with  singular potentials.

\begin{lemma}
Assume that {$\rho\in W_2^1\left((0, b_1)\cup (b_1, b)\right)$ and satisfies} \eqref{piecewiseac}. Then
\begin{align}\label{zxlambda}
z(x,\lambda):=(\rho(r))^{1/4}u(r,\lambda)
\end{align}
 satisfies the equation
\begin{align}
-&\frac{d z^{[1]}(x)}{dx}-
\sigma(x) z'(x)=\lambda z(x), \quad x\in(0,d)\cup (d,a),    \label{Aaa} \\
&z(d+)=d_1z(d-), \ z^{[1]}(d+)=d_1^{-1}z^{[1]}(d-),  \label{jump}
\end{align}
where $a$ is defined by \eqref{definitiona}, $\sigma\in L^2(0,a)$, $z^{[1]}(x)=z'(x)-\sigma(x) z(x),$
\begin{align}
d=\int_0^{b_1} \sqrt{\rho(t)} dt, \ d_1={b_2}^{1/4},  \label{definitione1}
\end{align}
\begin{align}
\sigma(x)=\frac{1}{4}\frac{\rho'(r)}{(\rho(r))^{3/2}}+g(x). \label{definitionsigma}
\end{align}
 Here $g(x)$ satisfies
\begin{align}\label{gprime}
g'(x)=\frac{1}{16}\frac{(\rho'(r))^2}{(\rho(r))^3}, \quad x\in (0,d)\cup (d,a),
\end{align}
and the jump condition
\begin{align}\label{gprime1}
g(d-)={b_2}^{1/2}g(d+).
\end{align}
\end{lemma}

\begin{proof} We first show that $z(x)$ satisfies \eqref{Aaa}.  For $x\in(0,d)\cup (d,a)$, by \eqref{zxlambda}, we know
	\begin{align}
	z'(x)=\frac{1}{4}(\rho(r))^{-5/4}\rho'(r)u(r)
	+(\rho(r))^{-1/4}u'(r).  \nonumber
	\end{align}
	 Hence
	\begin{align}
	z'(x)-\sigma(x) z(x)=(\rho(r))^{-1/4}u'(r)-g(x)(\rho(r))^{1/4}u(r), \label{zprime}
	\end{align}
\begin{align}
\sigma(x) z'(x)=&\frac{1}{16}(\rho'(r))^2(\rho(r))^{-11/4}u(r)
+\frac{1}{4}(\rho(r))^{-7/4}\rho'(r)u'(r)  \nonumber \\
&+\frac{1}{4}g(x)(\rho(r))^{-5/4}\rho'(r)u(r)+g(x)(\rho(r))^{-1/4}u'(r). \label{ac}
\end{align}	
Differentiating ~\eqref{zprime} with respect to~$x$, we obtain that
\begin{align}
-\frac{d z^{[1]}(x)}{dx}=&\frac{1}{4}(\rho(r))^{-7/4}\rho'(r)u'(r)-(\rho(r))^{-3/4}u''(r)+\frac{1}{16}(\rho'(r))^2(\rho(r))^{-11/4}u(r)  \nonumber\\
&+g(x)\left(\frac{1}{4}(\rho(r))^{-5/4}\rho'(r)u(r)
+(\rho(r))^{-1/4}u'(r)\right).      \label{z11111111}
\end{align}	
Subtracting ~\eqref{ac} from ~\eqref{z11111111}, by ~\eqref{zxlambda}, one has that
\begin{align}\nonumber
-\frac{d z^{[1]}(x)}{dx}-
\sigma(x) z'(x)=-(\rho(r))^{-\frac{3}{4}}u''(r)=\lambda z(x).
\end{align}
Therefore, we conclude that ~\eqref{Aaa} holds.

According to ~\eqref{piecewiseac},  \eqref{zxlambda} and ~\eqref{zprime}, $z(x,\lambda)$ satisfies the following jump condition
\begin{align}
z(d+)=d_1z(d-), \ z^{[1]}(d+)=d_1^{-1}z^{[1]}(d-)+d_2z(d-),  \label{jump1}
\end{align}
where
\begin{align}
d_1={b_2}^{1/4},
d_2= g(d-) {b_2}^{-{1}/{4}}-g(d+){b_2}^{{1}/{4}}.  \label{definitione2}
\end{align}
Assume that ~$d_2=0$.    By \eqref{definitione2}, ~$g(x)$  satisfies the jump condition~\eqref{gprime1}. The lemma is proved.
	\end{proof}

\begin{remark}
	Denote $q=\sigma', q\in W^{-1}_2(0,a).$ Then \eqref{Aaa} can be recast  in the form of SL equation $-z''+q(x)z=\lambda z$  in the distribution sense. Hence we call \eqref{Aaa}
	the SL equation with singular potentials.  We mention that Albeverio-Hryniv-Mykytyuk \cite{alb} showed that some SL operators in impedance form are unitarily equivalent to SL operators with singular potentials.
\end{remark}

By Liouville transformation \eqref{zxlambda}, we can transform \eqref{string} into discontinuous SL equation  with jump condition \eqref{jump}. Also
the characteristic function $D(\lambda)$ is transformed into
\begin{equation} \label{d}
D (\lambda)=\rho(b)^{1/4}\left\vert
\begin{array}{cc}
\frac{\sin (\sqrt{\lambda}b )}{\sqrt{\lambda}}&\beta z(a,\lambda) \\
\cos (\sqrt{\lambda}b )& z^{[1]}(a, \lambda)+g(a)z(a,\lambda) %
\end{array}%
\right\vert,
\beta=\frac{1}{\rho(b)^{1/2}}.
\end{equation}
Since $g(a)$ is an arbitrary real number, then we transform problem ~$Q(\rho)$ into a family of discontinuous ~SL problems with singular potentials.
In order to ensure the uniqueness of the image of ~Liouville transformation, we choose ~$g(a)=0$.
Therefore, we can transform problem ~$Q(\rho)$ into the problem
\begin{equation}\label{eq215}
\left\{
\begin{split}
& -\frac{d z^{[1]}(x)}{dx}-\sigma z'(x)=\lambda z(x), \quad x\in(0,d)\cup (d,a),     \\
&z(d+)=d_1z(d-), \ z^{[1]}(d+)=d_1^{-1}z^{[1]}(d-), \\
&z(0)=D(\lambda)=0,
\end{split}%
\right.
\end{equation}
where
\begin{equation}
D (\lambda)=\rho(b)^{1/4}\left\vert
\begin{array}{cc}
\frac{\sin (\sqrt{\lambda} b)}{\sqrt{\lambda}}&\beta z(a,\lambda) \\
\cos (\sqrt{\lambda}b )& z^{[1]}(a, \lambda)
\end{array}%
\right\vert.
\label{eq1.4=}
\end{equation}%

The following lemma shows that under the conditions that ~$\rho(b)$ is known and ~$g(a)=0$, the Liouville transformation is injective.

\begin{lemma} \label{unique}
	Assume that {$\rho\in W_2^1\left((0, b_1)\cup (b_1, b)\right)$ and satisfies} \eqref{piecewiseac},
	$\sigma(x)$, $d$, $d_1$ are defined by  \eqref{definitione1} and
	\eqref{definitionsigma}
	with $g(a)=0$. Then $\rho$ is uniquely determined by $d$, $d_1$, $\rho(b)$ and $\sigma(x),x\in[0,a]$.
\end{lemma}

\begin{proof}
	Denote $\hat{\rho}(x):=\rho(r)$. By \eqref{definitionsigma} and \eqref{gprime},  $(\hat{\rho}(x))^{1/4}$ and $ g(x)$ satisfy  ordinary differential equations
	\begin{align}\label{fangchengzu}
	\frac{d (\hat{\rho}(x))^{1/4}}{dx}&=(\hat{\rho}(x))^{1/4}(\sigma(x)-g(x)),\\
	\frac{d g(x)}{dx}&=(\sigma(x)-g(x))^2, \label{fangchengzu1}
	\end{align}
	and the following initial conditions
	\begin{align}\nonumber
	(\hat{\rho}(a))^{1/4}=\rho(b)^{1/4},
	g(a)=0.
	\end{align}
By the uniqueness theorem for ordinary differential equations,  we can uniquely determine ~$(\hat{\rho}(x))^{1/4}$, $g(x)$, $x\in (d,a)$ 	if ~$\rho(b)$ is known.
	Since ~$d$, $d_1$ are known, by  \eqref{piecewiseac}, \eqref{definitione1} and \eqref{gprime1}, we know that ~$b_2$, $(\hat{\rho}(d-))^{1/4}$ and $g(d-)$ are uniquely determined. Note that ~$(\hat{\rho}(x))^{1/4}$ and $g(x)$ also satisfy the ordinary differential equations ~\eqref{fangchengzu} and \eqref{fangchengzu1} on the interval ~$(0,d)$.
Therefore we can uniquely determine $(\hat{\rho}(x))^{1/4}$, $g(x)$, $x\in (0,d)$.
	Since~$x=\int_0^r\sqrt{\rho(s) }ds$, then $r(x)$ satisfies
\begin{align}\nonumber
\frac{dr}{dx}=\frac{1}{\sqrt{\hat{\rho}(x)}}
\end{align}
and $$r(0)=0.$$ By the uniqueness theorem for the differential equation, $\hat{\rho}(x)$ uniquely determines  $r(x)$ and hence ~$\rho(r)$ is uniquely determined.	
\end{proof}

\section{ Discontinuous Sturm-Liouville operator with singular potentials}
In this section, we consider the  SL equation \eqref{Aaa} with discontinuous condition  \eqref{jump}, denote it by $L(\sigma,d,d_1)$. Here $\sigma\in L^2(0,a), 0<d<a, d_1\ne 1>0.$

Let $s(x,\lambda) $ be the solution of \eqref{Aaa} satisfying the initial condition $s(0,\lambda) =0$, $ s^{[1]}(0,\lambda)=1$ and the discontinuous condition  \eqref{jump}.
By \eqref{Aaa}, for $0\le x<d$,  $s$ and $s^{[1]} $ satisfy the following equations (see also \cite{bondarenko})
\begin{align}
s&(x,\lambda)=\frac{\sin\sqrt{\lambda}x}{\sqrt{\lambda}}-
\int_0^x\frac{\sin\sqrt{\lambda} (x-t)}{\sqrt{\lambda}}\sigma(t)s^{[1]}(t,\lambda) dt  \nonumber \\
&+\int_0^x{\cos\sqrt{\lambda} (x-t)}\sigma(t)s(t,\lambda) dt
-\int_0^x\frac{\sin\sqrt{\lambda} (x-t)}{\sqrt{\lambda}}\sigma(t)^2s(t,\lambda) dt, \label{integrationz}
\end{align}	
\begin{align}
s&^{[1]}(x,\lambda)=\cos \sqrt{\lambda}x- \int_0^x\cos\sqrt{\lambda} (x-t)\sigma(t)s^{[1]}(t,\lambda) dt \nonumber\\
&-\sqrt{\lambda}\int_0^x{\sin\sqrt{\lambda} (x-t)}\sigma(t)s(t,\lambda) dt
-\int_0^x{\cos\sqrt{\lambda} (x-t)}\sigma(t)^2s(t,\lambda) dt. \label{integrantionz1aa}
\end{align}
We next show the equations that $s$ and $s^{[1]} $ satisfy for $d<x\le a$. Notice that for $d<x\le a$,  there exist $A,B\in \mathbb{R}$, so that
\begin{align}
s&(x,\lambda)=A\frac{\sin\sqrt{\lambda}(x-d)}{\sqrt{\lambda}}+B\cos \sqrt{\lambda}(x-d)
-\int_d^x\frac{\sin\sqrt{\lambda} (x-t)}{\sqrt{\lambda}}\sigma(t)s^{[1]}(t,\lambda) dt  \nonumber \\
&+\int_d^x{\cos\sqrt{\lambda} (x-t)}\sigma(t)s(t,\lambda) dt
-\int_d^x\frac{\sin\sqrt{\lambda} (x-t)}{\sqrt{\lambda}}\sigma(t)^2s(t,\lambda) dt, \label{integrationz11}
\end{align}	
\begin{align}
s&^{[1]}(x,\lambda)=A\cos \sqrt{\lambda}(x-d)-B\sqrt{\lambda}{\sin\sqrt{\lambda}(x-d)}-
\int_d^x\cos\sqrt{\lambda} (x-t)\sigma(t)s^{[1]}(t,\lambda) dt \nonumber\\
&-\sqrt{\lambda}\int_d^x{\sin\sqrt{\lambda} (x-t)}\sigma(t)s(t,\lambda) dt
-\int_d^x{\cos\sqrt{\lambda} (x-t)}\sigma(t)^2s(t,\lambda) dt. \label{integrantionz12}
\end{align}
By~\eqref{integrationz11}-\eqref{integrantionz12}, one has that ~$A=s^{[1]}(d+,\lambda), B=s(d+,\lambda)$.
On the other hand, From~\eqref{integrationz}, \eqref{integrantionz1aa} and the jump condition ~\eqref{jump}, we know that
\begin{align}
s&(d+,\lambda)=d_1\Big(\frac{\sin\sqrt{\lambda}d}{\sqrt{\lambda}}-\int_0^d
\frac{\sin\sqrt{\lambda} (d-t)}{\sqrt{\lambda}}\sigma(t)s^{[1]}(t,\lambda) dt  \nonumber \\
&+\int_0^d{\cos\sqrt{\lambda} (d-t)}\sigma(t)s(t,\lambda) dt
-\int_0^d\frac{\sin\sqrt{\lambda} (d-t)}{\sqrt{\lambda}}\sigma(t)^2s(t,\lambda) dt\Big), \label{integrationz1111}
\end{align}	
\begin{align}
s&^{[1]}(d+,\lambda)=d_1^{-1}\Big(\cos \sqrt{\lambda}d-
\int_0^d\cos\sqrt{\lambda} (x-t)\sigma(t)s^{[1]}(t,\lambda) dt \nonumber\\
&-\sqrt{\lambda}\int_0^d{\sin\sqrt{\lambda} (d-t)}\sigma(t)s(t,\lambda) dt
-\int_0^d{\cos\sqrt{\lambda} (d-t)}\sigma(t)^2s(t,\lambda)dt \Big).
\label{integrantionz111}
\end{align}
Therefore, for $d<x\le a$, $s$ and $s^{[1]}$ satisfy the following equations
\begin{align}
s&(x,\lambda)=\frac{1}{\sqrt{\lambda}}\left(d_1 \sin \sqrt{\lambda} d \cos \sqrt{\lambda} (x-d)+d_1^{-1} \cos \sqrt{\lambda} d \sin \sqrt{\lambda} (x-d)\right)\nonumber\\
&-\frac{1}{\sqrt{\lambda}}
\int_{0}^d \left(d_1 \sin \sqrt{\lambda}( d-t) \cos \sqrt{\lambda} (x-d)+d_1^{-1} \cos \sqrt{\lambda} (d-t) \sin \sqrt{\lambda}
(x-d)\right)\sigma(t)s^{[1]}(t)dt \nonumber \\
&+\int_{0}^d \left(d_1 \cos \sqrt{\lambda} (d-t) \cos \sqrt{\lambda} (x-d)-d_1^{-1} \sin \sqrt{\lambda}(d-t)
\sin \sqrt{\lambda} (x-d) \right)\sigma(t)s(t)dt \nonumber\\
&-\frac{1}{\sqrt{\lambda}}\int_{0}^d \left(d_1 \sin \sqrt{\lambda}( d-t) \cos \sqrt{\lambda} (x-d)+d_1^{-1} \cos
\sqrt{\lambda} (d-t) \sin \sqrt{\lambda} (x-d)\right)\sigma^2(t)s(t)dt \nonumber \\
&-\int_d^x\frac{\sin\sqrt{\lambda}(x-t)}{\sqrt{\lambda}}\sigma(t)s^{[1]}(t,\lambda) dt+\int_d^x{\cos\sqrt{\lambda} (x-t)}\sigma(t)z(t,\lambda) dt \nonumber \\
&-\int_d^x\frac{\sin\sqrt{\lambda} (x-t)}
{\sqrt{\lambda}}\sigma(t)^2s(t,\lambda) dt,  \label{integrationz1}
\end{align}	
\begin{align}
s&^{[1]}(x,\lambda)=\left(-d_1 \sin \sqrt{\lambda} d \sin \sqrt{\lambda} (x-d)
+d_1^{-1} \cos \sqrt{\lambda} d \cos \sqrt{\lambda} (x-d) \right)\nonumber\\
&+\int_{0}^d \left(d_1 \sin \sqrt{\lambda}( d-t) \sin \sqrt{\lambda} (x-d)-d_1^{-1} \cos \sqrt{\lambda}
(d-t) \cos \sqrt{\lambda} (x-d)\right)\sigma(t)s^{[1]}(t)dt \nonumber \\
&-\sqrt{\lambda}\int_{0}^d \left(d_1 \cos \sqrt{\lambda} (d-t)
\sin \sqrt{\lambda} (x-d)+d_1^{-1} \sin \sqrt{\lambda} (d-t) \cos \sqrt{\lambda} (x-d)\right)\sigma(t)s(t)dt \nonumber\\
&+\int_{0}^d \left(d_1 \sin \sqrt{\lambda}( d-t) \sin \sqrt{\lambda} (x-d)-d_1^{-1}
\cos \sqrt{\lambda} (d-t) \cos \sqrt{\lambda} (x-d)\right)\sigma^2(t)s(t)dt \nonumber \\
&-\int_d^x{\cos\sqrt{\lambda}(x-t)}\sigma(t)s^{[1]}(t,\lambda) dt -\sqrt{\lambda}\int_d^x{\sin\sqrt{\lambda} (x-t)}\sigma(t)s(t,\lambda) dt \nonumber \\
&-\int_d^x\cos\sqrt{\lambda}
(x-t)\sigma(t)^2s(t,\lambda) dt. \label{integrationz2}
\end{align}	

 Denote $Y(x,\lambda)=(z(x,\lambda), z^{[1]}(x,\lambda))^T$. For $ 0\le x<d$,  \eqref{integrationz} and \eqref{integrantionz1aa} can be written in the matrix form
\begin{align}\label{11111}
Y(x,\lambda)=Y_0(x,\lambda) +\int_0^x A(t,\lambda)Y(t) dt,\ 0\le x<d,
\end{align}
where
\begin{align}
Y_0(x,\lambda)=\begin{pmatrix}
\frac{\sin\sqrt{\lambda}x}{\sqrt{\lambda}}  \\
\cos \sqrt{\lambda}x
\end{pmatrix}, \nonumber
\end{align}
\begin{align}
A(t,\lambda)=\begin{pmatrix}
\cos\sqrt{\lambda} (x-t)\sigma(t)-\frac{\sin\sqrt{\lambda} (x-t)}{\sqrt{\lambda}}\sigma(t)^2 & -\frac{\sin\sqrt{\lambda} (x-t)}{\sqrt{\lambda}}\sigma(t) \\
-\sqrt{\lambda}{\sin\sqrt{\lambda} (x-t)}\sigma(t)-{\cos\sqrt{\lambda} (x-t)}\sigma(t)^2&-\cos\sqrt{\lambda} (x-t)\sigma(t)
\end{pmatrix}.  \nonumber
\end{align}

 Equation \eqref{11111} can be solved by the method of successive approximations; namely, with
\begin{align}\label{y0c}
Y_n(x,\lambda)\equiv
\begin{pmatrix}
Y_{1,n}(x,\lambda)  \\
Y_{2,n}(x,\lambda)
\end{pmatrix}
=\int_0^x A(t,\lambda)Y_{n-1}(t,\lambda) dt,
\end{align}
then at least formally, we have
\begin{align}
Y(x,\lambda)=\sum_{n=0}^\infty Y_n(x,\lambda).  \label {y}
\end{align}
 For $d<x\le a$, \eqref{integrationz1} and \eqref{integrationz2} can be written in the matrix form
\begin{align}\label{yn11}
Y(x,\lambda)=Y_0(x,\lambda) +\int_d^x A(t,\lambda)Y(t) dt,\ d<x\le a
\end{align}
with
\begin{align}
Y_0(x,\lambda)=\begin{pmatrix}
Y_{0,1}(x,\lambda)  \\
Y_{0,2}(x,\lambda)
\end{pmatrix}. \nonumber
\end{align}
Here $Y_{0,1}(x,\lambda)$ is the sum of the first four terms of \eqref{integrationz1},  $Y_{0,2}(x,\lambda)$ is the sum of the first four terms of \eqref{integrationz2}.
By the  method of successive approximations,  at least formally, $Y(x,\lambda)$ has the representation \eqref{y}.
Here
\begin{align}
Y_n(x,\lambda)\equiv
\begin{pmatrix}
Y_{n,1}(x,\lambda)  \\
Y_{n,2}(x,\lambda)
\end{pmatrix} \nonumber
=\int_d^x A(t,\lambda)Y_{n-1}(t,\lambda) dt.
\end{align}

	 Set
	\begin{align}\nonumber
		\varphi(x,\lambda)=
	\begin{cases}
	\frac{\sin\sqrt{\lambda }x}{\sqrt{\lambda }},  &0\le x<d, \\
	\frac{1}{\sqrt{\lambda}}\left(d_1 \sin \sqrt{\lambda} d \cos \sqrt{\lambda} (x-d)+d_1^{-1} \cos \sqrt{\lambda} d \sin \sqrt{\lambda} (x-d)\right), &d< x\le a.
	\end{cases}
	\end{align}
	and
		\begin{align}\nonumber
		\phi(x,\lambda)=
	\begin{cases}
 \cos \sqrt{\lambda }x,  &0\le x<d, \\
	-d_1 \sin \sqrt{\lambda} d \sin \sqrt{\lambda} (x-d)+d_1^{-1} \cos \sqrt{\lambda} d \cos \sqrt{\lambda} (x-d),&d< x\le a.
	\end{cases}
	\end{align}
Then we have the following lemma.	
	
\begin{lemma} \label{lemma22}
For all  $\lambda\in\mathbb{C}$ and $x\ne d$,  there exist $K(x, \cdot), N(x, \cdot) \in L^2(0,x)$, so that
\begin{align}\label{definitionkxt}
s(x,\lambda)=\varphi(x,\lambda)+\int_0^{x} K(x,t)\frac{\sin\sqrt{\lambda }t}
{\sqrt{\lambda }} dt,
\end{align}
\begin{align} \label{definitionnxt}
s^{[1]}(x,\lambda)=\phi(x,\lambda)+\int_0^{x} N(x,t)\cos \sqrt{\lambda}t dt.
\end{align}

\end{lemma}

\begin{proof}
We first show that if $0\le x<d$, for any $n\ge1 $, $Y_{1,n}$ and $Y_{2,n}$ have the following representation (see \eqref{y0c} for  definitions of $Y_{1,n}$ and $Y_{2,n}$)
\begin{align}
Y_{n,1}(x,\lambda)&=\int_0^x K_n(x,t)\frac{\sin\sqrt{\lambda}t}{\sqrt{\lambda}} dt,  \label{kn}\\
Y_{n,2}(x,\lambda)&=\int_0^x N_n(x,t)\cos\sqrt{\lambda}tdt. \label{nn}
\end{align}

First we calculate $Y_{1,1},Y_{1,2}$. By trigonometric addition
 formulas, using the change of variables and interchanging the order of integration,  we know that for $n=1$, \eqref{kn} and \eqref{nn} hold with
\begin{align}
K_1(x,t)=&\frac{1}{2}\sigma\left(\frac{x+t}{2}\right)-\frac{1}{2}\sigma\left(\frac{x-t}{2}\right)-\frac{1}{2}\int_0^t \sigma^2(s)ds \nonumber \\
&+\frac{1}{4}\int_t^x\sigma^2\left(\frac{\tau-t}{2}\right)-\sigma^2\left(\frac{\tau+t}{2}\right)d\tau,  \nonumber
\end{align}
\begin{align}
N_1(x,t)=&-\frac{1}{2}\sigma\left(\frac{x+t}{2}\right)-\frac{1}{2}\sigma\left(\frac{x-t}{2}\right)-\frac{1}{2}\int_0^t \sigma^2(s) ds \nonumber \\
&-\int_t^x \sigma^2(s) ds
+\frac{1}{4}\int_t^x\sigma^2\left(\frac{\tau-t}{2}\right)+\sigma^2\left(\frac{\tau+t}{2}\right)d\tau. \nonumber
\end{align}

Assume that for ~$n=j$,  ~\eqref{kn} and ~\eqref{nn} hold.  Letting ~$n=j+1$ in ~\eqref{y0c} and substituting the integral representation of ~$Y_{1,j},Y_{2,j}$ into
~\eqref{y0c}, one can see that  ~\eqref{kn} and \eqref{nn} hold for ~$n=j+1$,
where
\begin{align}
K_{j+1}&(x,t)=-\frac{1}{2}\int_{x-t}^x N_j(s,t-x+s)\sigma(s) ds
-\frac{1}{2}\int_{\frac{x-t}{2}}^{x-t} N_j(s,x-s-t)\sigma(s) ds \nonumber \\
&+\frac{1}{2}\int_{\frac{x+t}{2}}^x N_j(s,x-s+t)\sigma(s) ds-
\frac{1}{2}\int_t^x d\xi\Big(\int_{\xi-t}^\xi K_j(s,t-\xi+s)\sigma^2(s)ds  \nonumber \\
&-\int_{\frac{\xi-t}{2}}^{\xi-t} K_j(s,\xi-s-t)\sigma^2(s)ds+\int_{\frac{\xi+t}{2}}^\xi K_j(s,t-s+\xi)\sigma^2(s) ds\Big)\nonumber \\
&+\frac{1}{2}\int_{x-t}^x K_j(s,t-x+s)\sigma(s) ds-\frac{1}{2}\int_{\frac{x-t}{2}}^{x-t} K_j(s,x-s-t)\sigma(s) ds \nonumber \\
&+\frac{1}{2}\int_{\frac{x+t}{2}}^x K_j(s,x-s+t)\sigma(s) ds, \label{kn+1}
\end{align}
\begin{align}
N_{j+1}&(x,t)=-\frac{1}{2}\int_{x-t}^x N_j(s,t-x+s)\sigma(s) ds
-\frac{1}{2}\int_{\frac{x-t}{2}}^{x-t} N_j(s,x-s-t)\sigma(s) ds \nonumber \\
&-\frac{1}{2}\int_{\frac{x+t}{2}}^x N_j(s,x-s+t)\sigma(s) ds-
\frac{1}{2}\int_t^x d\xi\Big(\int_{\xi-t}^\xi K_j(s,t-\xi+s)\sigma^2(s)ds  \nonumber \\
&-\int_{\frac{\xi-t}{2}}^{\xi-t} K_j(s,\xi-s-t)\sigma^2(s)ds-\int_{\frac{\xi+t}{2}}^\xi K_j(s,\xi-s+t)\sigma^2(s) ds\Big)\nonumber \\
&-\int_t^x \sigma(s)^2 K_j(s,t)ds
+\frac{1}{2}\int_{x-t}^x K_j(s,t-x+s)\sigma(s) ds \nonumber \\
&-\frac{1}{2}\int_{\frac{x-t}{2}}^{x-t} K_j(s,x-s-t)\sigma(s) ds -\frac{1}{2}\int_{\frac{x+t}{2}}^x K_j(s,x-s+t)\sigma(s) ds. \label{nn+1}
\end{align}
By induction,
it follows that the series
\begin{align}
K(x,t):=\sum_{n=1}^{\infty} K_n(x,t), N(x,t):=\sum_{n=1}^{\infty} N_n(x,t) \nonumber
\end{align}
converge in $L^2(0,x)$ and hence $Y(x,\lambda)$ defined by \eqref{y} is indeed a solution of \eqref{11111}.  Moreover, ~$K(x,t)$ with respect to $t$ and the function ~$\sigma$ have the same smoothness.

We next show for ~$d<x\le a$, \eqref{definitionkxt} and ~\eqref{definitionnxt} hold. In this case, the computation is much more complicated.  We only present the main steps.
 By the definition of ~$Y_0(x,\lambda)$, one obtains that there exist ~$K_0(x,\cdot), N_0(x,\cdot)\in L^2(0,x)$, such that
\begin{align}
Y_{0,1}(x,\lambda)&=\frac{1}{\sqrt{\lambda}}\left(d_1 \sin \sqrt{\lambda} d \cos \sqrt{\lambda} (x-d)+d_1^{-1} \cos \sqrt{\lambda} d \sin \sqrt{\lambda} (x-d)\right) \nonumber\\
&+\int_0^x K_0(x,t)\frac{\sin\sqrt{\lambda}t}{\sqrt{\lambda}}dt,  \label{y011}
\end{align}
\begin{align}
Y_{0,2}(x,\lambda)&=\left(-d_1 \sin \sqrt{\lambda} d \sin \sqrt{\lambda} (x-d)+d_1^{-1} \cos \sqrt{\lambda} d \cos \sqrt{\lambda} (x-d)\right) \nonumber \\
&+\int_0^x  N_0(x,t)\frac{\sin\sqrt{\lambda}t}{\sqrt{\lambda}}dt.  \label{y021}
\end{align}

 We prove that for ~$d<x\le a$, \eqref{kn} and ~\eqref{nn} hold.  Denote 	\begin{align}\nonumber
\sigma_-(x) =
\begin{cases}
0, & 0<x<d,\\
\sigma(x),  & d<x\le a.
\end{cases}
\end{align}
Substituting ~\eqref{y011} and ~\eqref{y021} into ~\eqref{yn11} with $n=1$,   we get that  ~\eqref{kn} and ~\eqref{nn} hold for ~$n=1$. Assume that for ~$n=j$,  ~\eqref{kn} and ~\eqref{nn} hold.  Letting ~$n=j+1$ in ~\eqref{yn11} and substituting the integral representation of ~$Y_{1,j},Y_{2,j}$ into
~\eqref{y0c}, we can see that  ~\eqref{kn} and \eqref{nn} hold for ~$n=j+1$.
Here $K_{j+1}(x,t)$ and $N_{j+1}(x,t)$ are given by \eqref{kn+1} and \eqref{nn+1}, respectively, with the function $\sigma$ replaced by  $\sigma_-$.
 By  induction,  we obtain \eqref{kn} and ~\eqref{nn} hold for  ~$n\ge 1$.
Furthermore,  the series
 \begin{align}\nonumber
 K(x,t):=\sum_{n=0}^{\infty} K_n(x,t), N(x,t):=\sum_{n=0}^{\infty} N_n(x,t)
 \end{align}
  converge in $L^2(0,x)$ and hence $Y(x,\lambda)$ defined by \eqref{y} is indeed a solution of \eqref{yn11}. The proof is completed.
\end{proof}
\section{Weyl-Titchmarsh function}
Define the Weyl-Titchmarsh function of \eqref{Aaa} and \eqref{jump} by
\begin{align}\nonumber
m(x, \lambda)=-\frac{s^{[1]}(x-,\lambda)}{s(x-,\lambda)}.
\end{align}
By \eqref{definitionkxt} and \eqref{definitionnxt}, as $|\lambda|\rightarrow \infty$ in the sector
~$\Lambda_{\delta}:=\{\lambda\in \mathbb{C}| \delta<\arg(\lambda)<\pi-\delta, \delta\in (0,\pi/2)\}$,
$m(x,\lambda)$ has the asymptotic formula
\begin{align}
m(x,\lambda)=i\sqrt{\lambda}(1+o(1)),x\in(0,d)\cup (d,a).  \label{aaaaa}
\end{align}
Moreover, $m(x,\lambda)$ obeys the Riccati equation
\begin{align}
m'(x,\lambda)-m^2(x,\lambda)+2\sigma(x)m(x,\lambda)=\sigma^2(x)+\lambda, x\in (0,d)\cup (d,a),  \label{ricatti}
\end{align}
and the jump condition
\begin{align}\nonumber
m(d+,\lambda)=\frac{1}{d_1^{2}}m(d-,\lambda).
\end{align}

For given $y\in [0,a]$, let ${s}(x,\lambda;y),{c}(x,\lambda;y) $  be solutions of \eqref{Aaa} satisfying the  initial conditions
\begin{align}\label{normalized}
s(y,\lambda;y)=c^{[1]}(y,\lambda;y)=0, s^{[1]}(y,\lambda;y)=c(y,\lambda;y)=1
\end{align}
at $y$  and the jump condition ~\eqref{jump}.  Then
\begin{align}\label{langsiji}
W(c(x,\lambda;y),s(x,\lambda;y))\equiv c(x,\lambda;y)s^{[1]}(x,\lambda;y)-c^{[1]}(x,\lambda;y)s(x,\lambda;y)=1.
\end{align}
Obviously, we have  $s(x,\lambda)=s(x,\lambda;0)$ and
\begin{align}\label{linearcom}
s(x,\lambda)=& s^{[1]}(y,\lambda) s(x,\lambda;y)+s(y,\lambda) c(x,\lambda;y),\\
s^{[1]}(x,\lambda)=& s^{[1]}(y,\lambda) s^{[1]}(x,\lambda;y)+s(y,\lambda) c^{[1]}(x,\lambda;y). \label{linearcom1}
\end{align}
 For simplicity, denote $S(x,\lambda)\equiv s(x,\lambda;a),   C(x,\lambda)\equiv c(x,\lambda;a)$.

Define
\begin{align}\label{definitionpsi}
\Psi(x,\lambda)=\frac{s(x,\lambda)}{s(a,\lambda)}.
\end{align}
Then according to \eqref{linearcom},

\begin{align}\label{psi111111}
\Psi(x,\lambda)=C(x,\lambda)-{m(a, \lambda)}S(x,\lambda).
\end{align}
By \eqref{langsiji}, one has
\begin{align}\label{cccc111111111}
W(\Psi, S)=1.
\end{align}

By using the method of spectral mappings \cite{bondarenko,fre}, we obtain the following theorem. 
\begin{theorem}\label{uniqueness}
Assume that $d_1\ne1$. Then  Weyl-Titchmarsh function $m(a,\lambda)$ uniquely determines $ d, d_1$ and $\sigma(x), x\in [0,a]$.
\end{theorem}

\begin{proof}
	 We require that if a certain symbol $\gamma$ denotes an object related to $L(\sigma, d,d_1)$,
	then the corresponding symbol $\tilde{\gamma}$ denotes the analogous object related to $L(\tilde{\sigma}, \tilde{d},\tilde{d}_1)$.
	
Define the matrix $P(x,\lambda)=[P_{ij}(x,\lambda)]_{j,k=1,2}$ by the formula
\begin{align}\label{definitionp}
P(x,\lambda)\begin{pmatrix}
\tilde{S}(x,\lambda)& \tilde{\Psi}(x,\lambda) \\
\tilde{S}^{[1]}(x,\lambda)&\tilde{\Psi}^{[1]}(x,\lambda)
\end{pmatrix}
=\begin{pmatrix}
{S}(x,\lambda)& {\Psi}(x,\lambda) \\
{S}^{[1]}(x,\lambda)&{\Psi}^{[1]}(x,\lambda)
\end{pmatrix}.
\end{align}
Then from \eqref{cccc111111111},
\begin{align}\label{matrixp}
\begin{pmatrix}
{P_{11}}(x,\lambda)& {P_{12}}(x,\lambda) \\
{P_{21}}(x,\lambda)&P_{22}(x,\lambda)
\end{pmatrix}
=\begin{pmatrix}
-{S}\tilde{\Psi}^{[1]}+\tilde{S}^{[1]}{\Psi}&-\tilde{S}{\Psi}+{S}\tilde{\Psi}  \\
-{S}^{[1]}\tilde{\Psi}^{[1]}+\tilde{S}^{[1]}{\Psi}^{[1]}&-\tilde{S}{\Psi}^{[1]}+{S}^{[1]}\tilde{\Psi}
\end{pmatrix}.
\end{align}
By \eqref{definitionpsi} and Lemma \ref{lemma22},  as $|\lambda|\rightarrow \infty$ in the sector
~$\Lambda_{\delta}$, one obtains
\begin{align}\nonumber
|P_{11}(x,\lambda)|\le C, |P_{12}(x,\lambda)|=o(1).
\end{align}
If ~$m(a, \lambda)=\tilde{m}(a, \lambda)$, by ~\eqref{psi111111} and ~\eqref{matrixp}, $P_{11}(x,\lambda), P_{12}(x,\lambda)$ are entire functions with respect to ~$\lambda$. Using ~Phragm\'{e}n-Lindel\"{o}f theorem \cite[Section 6.1]{lev} and ~Liouville theorem, one has that ~$P_{11}(x,\lambda)=A(x), P_{12}(x,\lambda)=0$. Therefore,  by ~\eqref{definitionp},
\begin{align}\nonumber
S(x,\lambda)=A(x)\tilde{S}(x,\lambda),\Psi(x,\lambda)=A(x)\tilde{\Psi}(x,\lambda).
\end{align}
Since ~$W(\Psi, S)=W(\tilde{\Psi}, \tilde{S})=1$, we know that $A(x)^2=1$. From the asymptotic behavior of ~$S$ and ~$\tilde{S}$, we can get that
$A(x)= 1$. Therefore, ~$S(x,\lambda)=\tilde{S}(x,\lambda)$, $\Psi(x,\lambda)=\tilde{\Psi}(x,\lambda)$.
From the fact that  ~$S(0,\lambda)=-s(a,\lambda)$ and \eqref{definitionpsi}, for any ~$x\in [0,d)\cup (d,a]$, we have
\begin{equation}\label{sxlambdaequal}
s(x,\lambda)=\tilde{s}(x,\lambda).
\end{equation}

From \eqref{sxlambdaequal}, one obtains that $d=\tilde{d},d_1=\tilde{d}_1$.
By  equations
\begin{align}
-(s'-\sigma s)'-\sigma(s'-\sigma s)-\sigma^2s&=\lambda s, x\in (0,d)\cup (d,a), \nonumber \\
-(s'-\tilde{\sigma} s)'-\tilde{\sigma}(s'-\tilde{\sigma} s)-\tilde{\sigma}^2s&=\lambda s, \nonumber x\in (0,d)\cup (d,a),
\end{align}
one knows that for ~$x\in [0,d)\cup (d,a]$,  ~$((\sigma-\tilde{\sigma})s)'=(\sigma-\tilde{\sigma})s'$.
In particular, the function ~$(\sigma-\tilde{\sigma})s$ is absolutely continuous on the interval ~$[0, d)\cup(d,a]$.
Choosing ~$\lambda_0\in \mathbb{C}$ so that for any ~$x\in[0,a]$,
\begin{align}\label{sxlambda}
s(x, \lambda_0)\ne 0.
\end{align}
Then on the interval ~$[0, d)\cup (d,a]$, the function ~$(\sigma-\tilde{\sigma})$ is absolutely continuous and ~$(\sigma-\tilde {\sigma})'=0$ almost everywhere.
Therefore, there exist ~$C_1, C_2\in \mathbb{R}$, such that
\begin{align}\nonumber
\sigma-\tilde {\sigma}=
\begin{cases}
C_1, & 0\le x<d, \\
C_2, & d<x\le a.
\end{cases}
\end{align}
 According to ~$m(a, \lambda)=\tilde{m}(a, \lambda)$ and ~\eqref{sxlambdaequal},
$s^{[1]}(a,\lambda)=\tilde{s}^{[1]}(a,\lambda)$.
By definitions of $s^{[1]}$ and  $\tilde{s}^{[1]}$, we know that ~$$(\sigma(a)-\tilde{\sigma}(a))s(a,\lambda)=0.$$
Using ~\eqref{sxlambda},  one has  $C_2=0$.
Then one obtains that~$\sigma(x)= \tilde{\sigma}(x)$ on ~$(d,a]$.
Hence for any~$x\in (d,a]$, one can see that ~$s^{[1]}(x, \lambda)=\tilde{s}^{[1]}(x, \lambda)$.
Because~$d_1=\tilde{d}_1$, we have~$s^{[1]}(d-, \lambda)=\tilde{s}^{[1]}(d-, \lambda)$.
From ~\eqref{sxlambda} and the definition of $s^{[1]}$, one has ~$$\lim_{ x\rightarrow d-} \sigma(x)-\tilde {\sigma}(x)=0.$$ Then~$C_1=0$.
Hence, we know that  $\sigma(x)=\tilde{\sigma}(x)$ almost everywhere  on ~$[0,a]$.
\end{proof}

\begin{remark}
	Consider the equation \eqref{Aaa} with jump condition \eqref{jump1}, denote it by $L(\sigma, d, d_1,d_2)$. Let $\sigma_1(x)\equiv 0,$
	\begin{align}\nonumber
	\sigma_2(x) =
	\begin{cases}
	-2, & 0\le x<d,\\
	0,  & d<x\le a.
	\end{cases}
	\end{align}
	Then  for $d<x\le a$, $L(\sigma_1, d,2,1)$ and  $L(\sigma_2, d,2,0)$ have the same Weyl-Titchmarsh function
	\begin{align}\nonumber
	m(x,\lambda)=-\frac{A(\lambda)\cos\sqrt{\lambda}(x-d)-B(\lambda)\sqrt{\lambda}\sin\sqrt{\lambda}(x-d)}
	{A(\lambda)\sin\sqrt{\lambda}(x-d)/\sqrt{\lambda}+B(\lambda)\cos\sqrt{\lambda}(x-d)}.
	\end{align}
	Here
	\begin{align}
	A(\lambda)=\frac{\cos2\sqrt{\lambda}d}{2}+\frac{\sin\sqrt{\lambda}d}{\sqrt{\lambda}},
	B(\lambda)=\frac{2\sin\sqrt{\lambda}d}{\sqrt{\lambda}}. \nonumber
	\end{align}
	In order to ensure the uniqueness of the inverse spectral problem, we transform  $Q(\rho)$ into the equation ~\eqref{Aaa} with the jump condition ~\eqref{jump}.
\end{remark}

$m(a,\lambda)$ has the following high-energy asymptotic behavior.

\begin{lemma}\label{lemmahigh}
	Assume that $\sigma\in L^2(0,a)$ and $\sigma$ is $C^n$ near $a$ for some $n\in\mathbb{N}$, then 	as $|\lambda|\rightarrow \infty$ in the sector ~$\Lambda_{\delta}:=\{\lambda\in \mathbb{C}| \delta<\arg(\lambda)<\pi-\delta, \delta\in (0,\pi/2)\}$,
	 $m(a,\lambda)$ has an asymptotic formula
	\begin{align}
	m(a,\lambda)=i\sqrt{\lambda}+i\sum_{l=0}^{n} c_l(a)\frac{1}{\lambda^{l/2}}+o\left(\frac{1}{\lambda^{n/2}}\right). \label{asym}
	\end{align}
The expansion coefficients $c_l(a)$ can be
recursively computed from
	\begin{align}
	c_{0}(a)&=-i\sigma(a), c_1(a)=-\frac{1}{2}\sigma'(a),  \nonumber\\
	c_{l+1}(a)&=-\frac{i}{2}c_{l}'(a)-\frac{1}{2}
	\sum_{j=1
}^{l-1}c_j(a)c_{l-j}(a), l\ge 1. \label{recu}
	\end{align}
\end{lemma}
\begin{proof}
Assume that $\sigma$ is $C^n$ on the interval $[y,a]$.	
	We first compute the high-energy asymptotic form  of $-{s^{[1]}(a,\lambda;y)}/{s(a,\lambda;y)}$, where $s(a,\lambda;y)$  is normalized
	 according to \eqref{normalized}. By Lemma \ref{lemma22}, $s(x,\lambda;y)$ and $s^{[1]}(x,\lambda;y)$ have the following  representation
	\begin{align}
	s(x,\lambda;y)=\frac{\sin(\sqrt{\lambda }(x-y))}
	{\sqrt{\lambda }}+\int_y^{x} K(x,t;y)\frac{\sin(\sqrt{\lambda }(t-y))}
	{\sqrt{\lambda }} dt, \label{1}
	\end{align}
	\begin{align}
	s^{[1]}(x,\lambda;y)=\cos(\sqrt{\lambda }(x-y))+\int_y^{x} N(x,t;y)\cos (\sqrt{\lambda}(t-y) )dt. \label{2}
		\end{align}
Recall that \cite{cla,ryb} if $f$ is continuous on $[y,a]$, then as $|\lambda|\rightarrow \infty$ in the sector $\Lambda_{\delta}$,
\begin{align}
\int_y^a f(t)e^{-i\sqrt{\lambda}(t-y)} dt={e^{-i\sqrt{\lambda}(a-y)}}\left(-f(a)\frac{1}{i\sqrt{\lambda}}
+o(\frac{1}{\sqrt{\lambda}})\right). \label{reimann1}
\end{align}

	Since $\sigma$ is $C^n$ on  $[y,a]$, then kernel functions  $ K(a,t;y)$ and $ N(a,t;y)$ are also $C^n$ with respect to $t$ on  $[y,a]$.
	Letting $x=a$ in \eqref{1} and \eqref{2}, integration by  parts $n$ times,  using  \eqref{reimann1} and the estimates
	\begin{align}\nonumber
	\cos(\sqrt{\lambda}(a-y))=
	\frac{e^{-i\sqrt{\lambda}(a-y)}}{2}\left(1+O(e^{2i\sqrt{\lambda}(a-y)})\right),
	\end{align}
\begin{align}\nonumber
\sin(\sqrt{\lambda}(a-y))=-
\frac{e^{-i\sqrt{\lambda}(a-y)}}{2i}\left(1+O(e^{2i\sqrt{\lambda}(a-y)})\right),
\end{align}	
then there exist $m_l(a), \tau_l(a), l=0, \cdots, n$, so that
	\begin{align}
	s(a,  \lambda;y)=\frac{e^{-i\sqrt{\lambda}(a-y)}}{2\sqrt{\lambda}}
	\left(i+ \sum_{l=0}^{n}\frac{m_l(a)}{\lambda^{\frac{l+1}{2}}}
	+o(\lambda^{-\frac{n+1}{2}})\right),  \label{zay}
\end{align}
	\begin{align}
	s^{[1]}(a,  \lambda;y)=\frac{e^{-i\sqrt{\lambda}(a-y)}}{2\sqrt{\lambda}}
	\left(\sqrt{\lambda}+\sum_{l=0}^{n}\frac{\tau_l(a)}{\lambda^{\frac{l}{2}}}
	+o(\lambda^{-\frac{n}{2}})\right). \label{z1ay}
	\end{align}
	From \eqref{zay} and \eqref{z1ay}, one knows that
	\begin{align}
	 -\frac{s^{[1]}(a,  \lambda;y)}{s(a, \lambda;y)}&=-
	\left(\sqrt{\lambda}+\sum_{l=0}^{n}\frac{\tau_l(a)}{\lambda^{\frac{l}{2}}}
	+o(\lambda^{-\frac{n}{2}})\right)\times
	\left(i+ \sum_{j=1}^{n+1}\frac{m_{j-1}(a)}{\lambda^{\frac{j}{2}}}
	+o(\lambda^{-\frac{n+1}{2}})\right)^{-1} \nonumber \\
	&=i\sqrt{\lambda}+i\sum_{l=0}^{n} \hat{c}_l(a)\frac{1}{\lambda^{l/2}}+o\left(\frac{1}{\lambda^{n/2}}\right). \label{high1}
	\end{align}
		Substituting \eqref{high1}
	 into the Riccati equation \eqref{ricatti}, the
		  coefficients $\hat{c}_l(a), l=0,\cdots,n$, obey
		the recursion relation \eqref{recu}.

On the other hand,  by \eqref{langsiji}, \eqref{linearcom} and \eqref{linearcom1}, one has
\begin{align} \nonumber
\frac{s^{[1]}(a,\lambda;y)}{s(a,\lambda;y)}+m(a,\lambda)=\frac{s(y,\lambda)}{s(a,\lambda ) s(a,\lambda;y ) }.
\end{align}
From the integral representation of $s$, as~$|\lambda|\rightarrow\infty$ in the sector ~$\Lambda_{\delta}$,
\begin{align}\label{exponential}
\frac{s(y,\lambda)}{s(a,\lambda ) s(a,\lambda;y ) }=O(e^{-2(a-y)|\rm{Im} \sqrt{\lambda}|}).
\end{align}
Using ~\eqref{exponential}, we know that
 ~$m(a,\lambda)$ has a high-energy asymptotic expansion ~\eqref{asym} and its coefficients satisfy ~${c}_l(a)=\hat{c}_l(a), l=0,\cdots,n $.
Since ~$\hat{c}_l(a), l=0,\cdots,n, $  satisfy the recursion relation ~\eqref{recu}, then ~${c}_l(a), l=0,\cdots,n,$ also satisfy the recursion relation ~\eqref{recu}.
The proof is completed.   \end{proof}

\section{Inverse  problems by all eigenvalues}

In this section, we consider the inverse transmission problem knowing all eigenvalues. Let $u(r,\lambda)$ be the solution of  $-u''=\lambda\rho u$ satisfying initial conditions $u(0,\lambda)=0, u'(0,\lambda)=1.$  Then we have the following lemma.
\begin{lemma}
	Assume that {$\rho\in W_2^1\left((0, b_1)\cup (b_1, b)\right)$ and satisfies} \eqref{piecewiseac}. Then as $|\lambda|\rightarrow \infty$ in the sector ~$\Lambda_{\delta}:=\{\lambda\in \mathbb{C}| \delta<\arg(\lambda)<\pi-\delta, \delta\in (0,\pi/2)\}$,
	\begin{align}
	u(r,\lambda)=\frac{1}{\left(\rho(0)\rho(r)\right)^{1/4}\sqrt{\lambda}} \left(\sin(\sqrt{\lambda} x(r))+o(e^{{\rm{Im}}\sqrt{\lambda}x(r) })\right) ,  \label{urlambda}
	\end{align}
	\begin{align}
	u'(r,\lambda)=\left(\frac{\rho(r)}{\rho(0)}\right) ^{1/4}\left(\cos(\sqrt{\lambda} x(r))+o(e^{{\rm{Im}}\sqrt{\lambda}x(r) })\right)
	.  \label{asmptoticurlambda1}
	\end{align}
\end{lemma}

\begin{proof}
According to ~\eqref{zxlambda} and ~\eqref{zprime}, we know
\begin{align}\label{linearproperty}
s(x,\lambda)=z(x,\lambda)(\rho(0))^{1/4}.
\end{align}
In light of  \eqref{zxlambda}, ~\eqref{zprime} and ~\eqref{linearproperty},  an application of  Lemma ~\ref{lemma22} and Riemann-Lebesgue lemma yields ~\eqref{urlambda} and ~\eqref{asmptoticurlambda1}.
\end{proof}

It is known that $\rho$ is uniquely determined by the knowledge of two sets of spectra~\cite{eck, kre, kre1}. If $\rho$ satisfies~\eqref{piecewiseac}, we provide a new proof of two-spectra theorem.
\begin{lemma} \label{lemma3.1}
	Assume that {$\rho\in W_2^1\left((0, b_1)\cup (b_1, b)\right)$ and satisfies} ~\eqref{piecewiseac}.  Then all zeros of $u(b,\lambda)$ and $u'(b,\lambda)$ uniquely determine $\rho(r)$ on $[0,b]$.
\end{lemma}
\begin{proof}
	First note that by (11.7) in \cite{kre2}, the constant $a$ is uniquely determined by all zeros of $u(b,\lambda)$.
According to ~\eqref{zxlambda} and the requirement $g(a)=0$, we get \begin{align}
m(a,\lambda)=-\frac{u'(b,\lambda)}{\beta u(b,\lambda)}. \label{aaaxx}
\end{align}
Here $\beta$ is defined by \eqref{d}. By
\begin{align}\label{ub0}
u(b,0)=b,\ u'(b,0)=1, \end{align}
we know that all zeros of $u(b,\lambda)$ and $u'(b,\lambda)$ uniquely determine $u(b,\lambda)$ and $u'(b,\lambda)$, respectively.
From ~\eqref{aaaaa} and  \eqref{aaaxx}, we have that
$ \rho(b)$ and hence ~$m(a,\lambda)$ are uniquely determined by all zeros of ~$u(b,\lambda)$ and ~$u'(b,\lambda)$.
Using Theorem \ref{uniqueness} and Lemma ~\ref{unique}, all zeros of $u(b,\lambda)$ and $u'(b,\lambda)$ uniquely determine $\rho(r)$ on $[0,b]$.
\end{proof}

\begin{lemma} \label{dlambda1}
Assume that {$\rho\in W_2^1\left((0, b_1)\cup (b_1, b)\right)$ and satisfies} ~\eqref{piecewiseac}.\\
{\rm{(i)}} If ~$\rho(b)\ne1$, then there exists ~$A_0>0$, so that in the sector ~$\Lambda_{\delta}:=\{\lambda\in \mathbb{C}| \delta<\arg(\lambda)<\pi-\delta, \delta\in (0,\pi/2)\}$, $D(\lambda)$ has the following estimate
\begin{align}
|D(\lambda)|\ge A_0\frac{e^{|{\rm{Im} }\sqrt{\lambda}|(a+b)}}{|\sqrt{\lambda}|}, |\lambda|\rightarrow\infty. \nonumber
\end{align}
 {\rm{(ii)}}     Assume that ~$\rho(b)=1$ and there exist ~$m\ge 1, \varepsilon>0$, so that ~$\rho\in C^{(m)}(b-\varepsilon,b]$,
for~$k=1,\cdots, m-1$, $\rho^{(k)}(b)=0$ and~$\rho^{(m)}(b)\ne0$.
Then there exists ~$A_0>0$, so that in the sector~$\Lambda_{\delta}$,
\begin{align} \label{D}
|D(\lambda)|\ge A_0\frac{e^{|{\rm{Im} }\sqrt{\lambda}|(a+b)}}{|\sqrt{\lambda}|^{m+1}}, |\lambda|\rightarrow\infty.
\end{align}

\end{lemma}
\begin{proof}
We first prove that (ii) holds. According to \eqref{definitionsigma}, $\sigma \in C^{(m-1)}(a-\varepsilon,a]$ and for $=1,\cdots, m-2$,  $\sigma^{(l)}(a)=0$. From \eqref{eq1.4=}, we know that
\begin{align}
D(\lambda)=-\frac{\sin \sqrt{\lambda}b}{\sqrt{\lambda}}z(a,\lambda)\left(\frac{\sqrt{\lambda}\cos \sqrt{\lambda}b}{\sin \sqrt{\lambda}b}+m(a,\lambda)\right). \label{dlambda}
\end{align}
Notice that
in the sector  ~$\Lambda_{\delta}$, for any ~$p\in \mathbb{N}$, one has
\begin{align}\nonumber
-\frac{\sqrt{\lambda}\cos \sqrt{\lambda}b}{\sin \sqrt{\lambda}b}=i\sqrt{\lambda}+o(\frac{1}{\lambda^{p/2}}), |\lambda|\rightarrow\infty.
\end{align}
From the high-energy asymptotics of ~$m(a,\lambda)$, we know that there exists ~$A_0>0$, so that in the sector ~$\Lambda_{\delta}$,
\begin{align}\nonumber
\left|\frac{\sqrt{\lambda}\cos \sqrt{\lambda}b}{\sin \sqrt{\lambda}b}+m(a,\lambda)\right|\ge A_0 \frac{1}{|\sqrt{\lambda}|^{m-1}}, |\lambda|\rightarrow\infty.
\end{align}
Therefore by ~\eqref{dlambda}, one can obtain \eqref{D}.

We next show that~(i) holds. From ~\eqref{eq1.4=}, one knows
\begin{align}
\frac{D(\lambda)}{\rho(b)^{1/4}}&=(1-\beta)\frac{\sin \sqrt{\lambda}b}{\sqrt{\lambda}}z^{[1]}(a,\lambda)+ \beta\frac{\sin \sqrt{\lambda}b}{\sqrt{\lambda}}z(a,\lambda)\left(-\frac{\sqrt{\lambda}\cos \sqrt{\lambda}b}{\sin \sqrt{\lambda}b}-m(a,\lambda)\right) \nonumber\\
&\equiv D_1+D_2. \nonumber
\end{align}
By the asymptotic form of ~$z^{[1]}(a,\lambda)$, one obtains  there exists ~$A_0>0$, so that in the sector ~$\Lambda_{\delta}$,
\begin{align}
|D_1(\lambda)|\ge A_0\frac{e^{|{\rm{Im}} \sqrt{\lambda}|(a+b)}}{|\sqrt{\lambda}|}, |\lambda|\rightarrow\infty. \label{d1}
\end{align}
By \eqref{aaaaa},  in the sector ~$\Lambda_{\delta}$, $D_2(\lambda)$ has the following estimate
\begin{align}
D_2(\lambda)= O{(|\lambda|^{-1})e^{|{\rm{Im} }\sqrt{\lambda}|(a+b)}}o(|\lambda|^{1/2})=o{(|\lambda|^{-1/2})e^{|{\rm{Im}} \sqrt{\lambda}|(a+b)}}, |\lambda|\rightarrow\infty. \label{d2}
\end{align}
According to ~\eqref{d1} and ~\eqref{d2}, we can obtain (i).
\end{proof}

When $a<b$, we prove the following uniqueness theorem.

\begin{theorem}\label{a<b}
	Assume that {$\rho\in W_2^1\left((0, b_1)\cup (b_1, b)\right)$ and satisfies} \eqref{piecewiseac} and $a<b$. Then all special  transmission eigenvalues uniquely determine $\rho$.
\end{theorem}
\begin{proof}
We require that if a certain symbol $\gamma$ denotes an object related to $Q( \rho)$,
then the corresponding symbol $\tilde{\gamma}$ denotes the analogous object related to $Q(\tilde{\rho})$.

	From \eqref{eq1.5}, we know
	\begin{align}\nonumber
	\frac{1}{\gamma}D\left(\frac{k^2\pi^2}{b^2}\right)=\frac{1}{\tilde{\gamma}}\tilde{D}\left(\frac{k^2\pi^2}{b^2}\right), k\in \mathbb{N}.
	\end{align}
Then by \eqref{eq1.4},
	\begin{align}
	\frac{1}{\gamma}u\left(b, \frac{k^2\pi^2}{b^2}\right)=\frac{1}{\tilde{\gamma}}\tilde{u}\left(b, \frac{k^2\pi^2}{b^2}\right), k\in \mathbb{N}. \label{u}
	\end{align}
	Define
	\begin{align}\nonumber
	f_1(\lambda)=\frac{1}{\gamma}u(b, \lambda)-\frac{1}{\tilde{\gamma}}\tilde{u}(b, \lambda).
	\end{align}
	From ~\eqref{u}, ~$\frac{\sqrt{\lambda}f_1(\lambda)}{\sin \sqrt{\lambda}b}$ is an entire function. Since ~$a<b$,
	using ~\eqref{urlambda}, we obtain that in the sector~$\Lambda_{\delta}$,
	\begin{align}\nonumber
	\lim_{|\lambda|\rightarrow \infty}\frac{\sqrt{\lambda}f_1(\lambda)}{\sin \sqrt{\lambda}b}=0.
	\end{align}
	According to ~Phragm\'{e}n-Lindel\"{o}f theorem and ~Liouville theorem, one has $\frac{\sqrt{\lambda}f_1(\lambda)}{\sin \sqrt{\lambda}b}\equiv 0$. Then ~$f_1(\lambda)\equiv 0,$ and hence
	\begin{align}\label{ublambda1111}
	\frac{1}{\gamma}u(b, \lambda)=\frac{1}{\tilde{\gamma}}\tilde{u}(b, \lambda).
	\end{align}
	
	By a similar argument, from
	\begin{align}\nonumber
	\frac{1}{\gamma}{D}\left(\frac{(2k-1)^2\pi^2}{4b^2}\right)=\frac{1}{\tilde{\gamma}}\tilde{D}\left(\frac{(2k-1)^2\pi^2}{4b^2}\right), k\in \mathbb{N},
	\end{align}
	one has
	\begin{align}
	\frac{1}{\gamma}u'\left(b, \frac{(2k-1)^2\pi^2}{4b^2}\right)=\frac{1}{\tilde{\gamma}}\tilde{u}'\left(b, \frac{(2k-1)^2\pi^2}{4b^2}\right), k\in \mathbb{N}. \label{uprime}
	\end{align}
	Define
	\begin{align}\nonumber
	f_2(\lambda)=\frac{1}{\gamma}u'(b, \lambda)-\frac{1}{\tilde{\gamma}}\tilde{u}'(b, \lambda).
	\end{align}
	From ~\eqref{uprime} we know that $\frac{f_2(\lambda)}{\cos \sqrt{\lambda}b}$ is an entire function. Since ~$a<b$,
	according to ~\eqref{asmptoticurlambda1}, in the sector ~$\Lambda_{\delta}$, we have
	\begin{align}\nonumber
	\lim_{|\lambda|\rightarrow \infty}\frac{f_2(\lambda)}{\cos \sqrt{\lambda}b}=0.
	\end{align}
	By using ~Phragm\'{e}n-Lindel\"{o}f theorem and ~Liouville theorem, one obtains
	$\frac{f_2(\lambda)}{\cos \sqrt{\lambda}b}\equiv 0$. Then ~$f_2(\lambda)\equiv 0$, and hence
	\begin{align}\label{ublambda11111}
	\frac{1}{\gamma}u'(b, \lambda)=\frac{1}{\tilde{\gamma}}\tilde{u}'(b, \lambda).
	\end{align}
	By~\eqref{ublambda1111}, \eqref{ublambda11111} and Lemma~\ref{lemma3.1}, we know that~$\rho\equiv\tilde{\rho}$. The proof is complete.
	\end{proof}

When $a=b$, we need more information to uniquely determine $\rho$.

\begin{theorem}\label{a=b}
	Assume that {$\rho\in W_2^1\left((0, b_1)\cup (b_1, b)\right)$ and satisfies} \eqref{piecewiseac} and $a=b$.  Assume that
	one of the following conditions holds:
	\par	{\rm{(i)}}	 the constant $\gamma$ in \eqref{eq1.5} is known;
	\par	{\rm{(ii)}}  $\rho(b)\ne1$  is known;
	\par	{\rm{(iii)}}  $\rho(b)=1$ and $\rho\in C^{(m)}(b-\varepsilon,b]$ for some $\varepsilon>0$ and some $m\in \mathbb{N}$,   for $k=1,\cdots, m-1$,  $\rho^{(k)}(b)=0$  and $\rho^{(m)}(b)\ne0$ is known. \\
	Then  all special transmission eigenvalues uniquely determine $\rho$.
\end{theorem}

\begin{proof} We first prove (i).
Since~$a=b$, arguing as in Theorem ~\ref{a<b}, one can obtain that in the sector~$\Lambda_{\delta}$,
~${\sqrt{\lambda}f_1(\lambda)}/({\sin \sqrt{\lambda}b}), {f_2(\lambda)}/({\cos \sqrt{\lambda}b})$ are bounded.
By
~Phragm\'{e}n-Lindel\"{o}f theorem and ~Liouville theorem, then there exist ~$C_1, C_2\in \mathbb{R}$, so that
\begin{align}
\frac{1}{\gamma}u(b, \lambda)-\frac{1}{\tilde{\gamma}}\tilde{u}(b, \lambda)=C_1 \frac{\sin\sqrt{\lambda}b}{\sqrt{\lambda}}, \nonumber \\
\frac{1}{\gamma}u'(b, \lambda)-\frac{1}{\tilde{\gamma}}\tilde{u}'(b, \lambda)=C_2 \cos\sqrt{\lambda} b.\nonumber
\end{align}
Letting~$\lambda=0$, by~\eqref{ub0}, we know~$C_1=C_2=\frac{1}{\gamma}
-\frac{1}{\tilde{\gamma}}.$ Since~$\gamma=\tilde{\gamma}$, then~$C_1=C_2=0$. According to Lemma ~\ref{lemma3.1}, $\rho\equiv\tilde{\rho}$. (i) is proved.

We next show ~(ii). Define
\begin{align}
H(\lambda)=\beta (z(a,\lambda)\tilde{z}^{[1]}(a,\lambda)-\tilde{z}(a,\lambda)z^{[1]}(a,\lambda)), F(\lambda)=\frac{H(\lambda)}{D(\lambda)}.\label{definitionhz}
\end{align}
  Since~$\rho(b)=\tilde{\rho}(b)$, by~\eqref{eq1.4=}, we know
\begin{align}\label{1a}
H(\lambda)=\frac{1}{\rho(b)^{1/4}\cos{\sqrt{\lambda}b}}(\tilde{D}(\lambda)z^{[1]}(a,\lambda)-D(\lambda)\tilde{z}^{[1]}(a,\lambda)).
\end{align}

  Assume that ~$\mu_m$ is a  zero of ~$D(\lambda), \tilde{D}(\lambda)$ of multiplicity $k$ satisfying ${\cos{b\sqrt{\mu_m}}}\ne0$. From \eqref{1a},
~$\mu_m$ is a  zero of ~$H(\lambda)$ of multiplicity $k$.  In this case,
$F(\lambda)$ is an entire function.

Assume that ~$\mu_m$ is a  zero of ~$D(\lambda), \tilde{D}(\lambda)$ of multiplicity $k$ satisfying ${\cos{b\sqrt{\mu_m}}}=0$.
By ~\eqref{eq1.4=} and the fact that
\begin{align}\nonumber
D(\mu_m)=\tilde{D}(\mu_m)=0,
\end{align}
one has
\begin{align}
z^{[1]}(a,\mu_m)=\tilde{z}^{[1]}(a,\mu_m)=0. \nonumber
\end{align}
Then~$\mu_m$ is also a  zero of ~$H(\lambda)$ of multiplicity $k$. Therefore in this case, we conclude that  $F(\lambda)$ is also an entire function.

  From ~\eqref{definitionhz}, one can obtain
  \begin{align}\label{ko}
  H(\lambda)=\beta z(a,\lambda)\tilde{z}(a,\lambda)(\tilde{m}(a,\lambda)-m(a,\lambda)).
  \end{align}
  Using ~\eqref{aaaaa}, as $|\lambda|\rightarrow\infty$ in the sector ~$\Lambda_{\delta}$, we have
  \begin{align}\label{ko1}
  \tilde{m}(a,\lambda)-m(a,\lambda)=o(|\lambda|^{1/2}).
  \end{align}
  By Lemma ~\ref{lemma22}, \eqref{ko} and ~\eqref{ko1}, as $|\lambda|\rightarrow\infty$ in the sector ~$\Lambda_{\delta}$, one gets
  \begin{align}\nonumber
  H(\lambda)=O(|\lambda|^{-1}e^{2a|\rm{Im }\sqrt{\lambda}|})o(|\lambda|^{1/2})=o(|\lambda|^{-1/2}e^{2a|\rm{Im } \sqrt{\lambda}|}).
  \end{align}
  According to Lemma~\ref{dlambda1}, in the sector~$\Lambda_{\delta}$,
  \begin{align}\nonumber
  F(\lambda)=o(|\lambda|^{-1/2}e^{2a|\rm{Im} \sqrt{\lambda}|})O(|\lambda|^{1/2}e^{-2a|\rm{Im} \sqrt{\lambda}|})=o(1), |\lambda |\rightarrow \infty.
  \end{align}
  By Phragm\'{e}n-Lindel\"{o}f theorem and Liouville theorem, we have ~$F(\lambda)\equiv0$.
  Therefore, we have ~$m(a,\lambda)\equiv \tilde{m}(a,\lambda)$. According to Theorem
  ~\ref{uniqueness} and Lemma~\ref{unique}, one knows that $\rho\equiv \tilde{\rho}.$ (ii) is proved.

  Finally, we prove ~(iii). Let ~$\rho(b)=1$ and there exist ~$m\ge 1, \varepsilon>0$,
  so that~$\rho\in C^{(m)}(b-\varepsilon,b]$, for~$k=1,\cdots, m-1$,
  $\rho^{(k)}(b)=\tilde{\rho}^{(k)}(b)=0$ and~$\rho^{(m)}(b)=\tilde{\rho}^{(m)}(b)\ne0$ are known.
  From~\eqref{definitionsigma}, one can see that $\sigma, \tilde{\sigma}\in C^{(m-1)}(a-\varepsilon,a]$ and
  $\sigma^{(l)}(a)=\tilde{\sigma}^{(l)}(a), l=1,\cdots, m-1$.
  From the high-energy asymptotic form ~\eqref{asym}  of the  Weyl-Titchmarsh function, one obtains that as $|\lambda|\rightarrow\infty$ in the sector~$\Lambda_{\delta}$,
  \begin{align}\label{ko2}
  \tilde{m}(a,\lambda)-m(a,\lambda)=o(|\lambda|^{-m/2+1/2}).
  \end{align}
  By Lemma~\ref{lemma22}, \eqref{ko} and ~\eqref{ko2},  as $|\lambda|\rightarrow\infty$ in the sector ~$\Lambda_{\delta}$,
  \begin{align}\nonumber
  H(\lambda)=O(|\lambda|^{-1}e^{2a|\rm{Im} \sqrt{\lambda}|})o(|\lambda|^{-m/2+1/2})=o(|\lambda|^{-m/2-1/2}e^{2a|\rm{Im } \sqrt{\lambda}|}).
  \end{align}
  Using Lemma~\ref{dlambda1}, in the sector ~$\Lambda_{\delta}$,
  \begin{align}\nonumber
  F(\lambda)=o(|\lambda|^{-m/2-1/2}e^{2a|\rm{Im }\sqrt{\lambda}|})O(|\lambda|^{m/2+1/2}e^{-2a|\rm{Im} \sqrt{\lambda}|})=o(1), |\lambda |\rightarrow \infty.
  \end{align}
  According to the Phragm\'{e}n-Lindel\"{o}f theorem and Liouville theorem, we have ~$F(\lambda)\equiv0$. Therefore ~$m(a,\lambda)\equiv \tilde{m}(a,\lambda)$. By Theorem
  ~\ref{uniqueness} and Lemma~\ref{unique}, one concludes  ~$\rho\equiv \tilde{\rho}.$ (iii)  is proved.
\end{proof}

 \section{Inverse problems by almost real subspectrum}

In this section, we study properties of ``almost real subspectrum'' $\{\mu_m\}_{m=1}^\infty$ {\cite{but2,mcl}} and  recover  the refractive index from the ``almost real subspectrum'' $\{\mu_m\}_{m=1}^\infty$  and partial information on the refractive index.
 We always assume that ~$\rho(b)=1$ in this section.

  We need the following lemma.

 \begin{lemma} \label{pt}
 	Assume that complex numbers ~$a_{mn}(m,n\ge1)$ satisfy
 	\begin{align}
 	|a_{mn}|=O\left(\frac{m\beta_m}{m^2-n^2}\right), m\ne n,  \label{amn}
 	\end{align}
 	where~$\{\beta_m\}_{m=1}^\infty\in \ell^2$.
 	Then there exists ~$\{\gamma_n\}_{n=1}^\infty\in \ell^2$, such that
 	\begin{align}
 	\prod_{m\ge1, m\ne n}(1+a_{mn}) =1+O(\gamma_n)=1+o(1), n\ge1.   \label{l2}
 	\end{align}
 In particular, if~$\beta_m=O(1/m)$,  then $\gamma_n=O(\log n/n)$.
 \end{lemma}

 \begin{proof}
 	If~$\beta_m=O(1/m)$, $\gamma_n=O(\log n/n)$ comes from Lemma E.1 in~\cite{pos}.
 	
 	We first prove that if ~$\{\beta_m\}_{m=1}^\infty\in \ell^2$, then
 	\begin{align}\label{fuliye}
 	\begin{Bmatrix}
 	\sum_{m\ge1, m\ne n} \frac{\beta_m}{|m-n|}
 	\end{Bmatrix}_{n=1}^\infty
 	\in \ell^2.
 	\end{align}
 	To this end, it suffices to show that the infinite matrix
 	\begin{align}
 	A=
 	\begin{pmatrix}
 	0   & 1    & \frac{1}{2}    & \frac{1}{3}  &  \frac{1}{4}  &\cdots       \\
 	1    & 0   & 1  & \frac{1}{2} &   \frac{1}{3} &\cdots    \\
 	\frac{1}{2}    & 1    & 0  & 1 &  \frac{1}{2}  & \cdots    \\
 	\frac{1}{3} &\frac{1}{2}  &1& 0 & 1& \cdots   \\
 	\frac{1}{4}      & \frac{1}{3}      & \frac{1}{2}    & 1 & 0&\cdots \\
 	\cdots&\cdots&\cdots&\cdots&\cdots&\cdots
 	\end{pmatrix}\nonumber
 	\end{align}
 	is a bounded linear operator from $\ell^2$ to $\ell^2$.  Let ${e}_m$ be the sequence in $\ell^2$ which has all its terms equal  to zero except for a one  in the $m$-th place.  Obviously,  the $n$-th place for ~$A{e}_m$ satisfies
 	\begin{align}\nonumber
 	(A{e}_m)_n=\begin{cases}
 	0,  & m=n, \\
 	\frac{1}{|m-n|}, & m\ne n.
 	\end{cases}
 	\end{align}
 	Hence  \begin{align}\nonumber
 	||A{e}_m||^2=\sum_{n=1,n\ne m}^\infty \frac{1}{{(m-n)}^2}\le 2\sum_{m=1}^\infty \frac{1}{{m}^2}=\frac{\pi^2}{3}.
 	\end{align}
 	Therefore~$A$ is a bounded linear operator from~${\rm{span}}\{{e}_1,{e}_2,\cdots\}$ to~$\ell^2$ with norm ~$||A||\le \pi/\sqrt{3}$. Since ${\rm{span}}\{{e}_1,{e}_2,\cdots\}$ is dense in $\ell^2$, then~$A$ is a bounded linear operator from~$\ell^2$ to $\ell^2$ with  norm ~$||A||\le \pi/\sqrt{3}$.
 	
 	We next prove that ~\eqref{l2} holds. By~\eqref{amn}, there exists $C>0$, so that
 	\begin{align}
 	\sum_{m=1,m\ne n}^\infty |a_{mn}|\le C \sum_{m=1,m\ne n}^\infty\left|\frac{m\beta_m}{m^2-n^2}\right|. \label{amn1}
 	\end{align}
 	Notice that
 	\begin{align}
 	\sum_{m\ge1, m\ne n}\left|\frac{m\beta_m}{m^2-n^2}\right|&\le \sum_{m\ge1, m\ne n}\left|\frac{\beta_m}{m-n}\right|. \label{po}
 	\end{align}
 	By~\eqref{fuliye}, \eqref{amn1} and~\eqref{po}, one has
 	\begin{align}
 	\begin{Bmatrix}
 	\sum_{m=1,m\ne n}^\infty |a_{mn}|
 	\end{Bmatrix}_{n=1}^\infty
 	\in \ell^2.
 	\end{align}
According to the  inequality
 	\begin{align}
 	\left|\prod_{m\ge1, m\ne n}(1+a_{mn})-1\right|&\le \prod_{m\ge1, m\ne n} (1+|a_{mn}|)-1 \nonumber \\
 	&\le e^{\sum_{m\ge1, m\ne n}|a_{mn}|}-1 \nonumber \\
 	&=O\left(\sum_{m\ge1, m\ne n}|a_{mn}|\right), \nonumber
 	\end{align}
 	we conclude that
 	\begin{align}\nonumber
 	\begin{Bmatrix}
 	\prod_{m\ge1, m\ne n}(1+a_{mn})-1
 	\end{Bmatrix}_{n=1}^\infty
 	\in \ell^2.
 	\end{align}
 	The lemma is proved.
 \end{proof}

Consider the function
\begin{align}\label{d0}
D_0(\lambda)=\alpha_1 \frac{\sin \sqrt{\lambda}(b-a)}{\sqrt{\lambda}}-\alpha_2 \frac{\sin \sqrt{\lambda} \xi}{\sqrt{\lambda}},
\end{align}
where ~$$\alpha_1=(d_1+d_1^{-1})/2, \alpha_2=(d_1-d_1^{-1})/2,  \xi=2d-a+b.$$ For~$m\in\mathbb{N} $, denote
\begin{align}\nonumber
x_{1,m}=\frac{(m\pi-\arcsin \frac{\alpha_0}{\alpha_1})^2}{(a-b)^2}, x_{2,m}=\frac{(m\pi+\arcsin \frac{\alpha_0}{\alpha_1})^2}{(a-b)^2},
\end{align}
where $\alpha_0=\max\{d_1,d_1^{-1}\}. $ By~\eqref{d0},
\begin{align}
D_0(x_{1,m})D_0(x_{2,m})=\frac{(-1)^{m}\alpha_0-\alpha_2\sin \sqrt{x_{1,m}}\xi}{\sqrt{x_{1,m}}}  \frac{(-1)^{m+1}\alpha_0-\alpha_2\sin \sqrt{x_{2,m}} \xi}{\sqrt{x_{2,m}}}<0. \nonumber
\end{align}
Hence for any ~$m\in \mathbb{N}$, ~$D_0(\lambda)$  has at least one zero $\mu_{0,m}$ on the interval~$(x_{1,m}, x_{2,m})$, namely
\begin{align}
\mu_{0,m}\in (x_{1,m}, x_{2,m}). \label{large}
\end{align}

We next show if ~$|\xi|\le |a-b|$, then for any ~$m\in\mathbb{N}$, $\mu_{0,m}$ is a simple zero of ~$D_0(\lambda)$.
Denote~$k=\sqrt{\lambda}, k_m=\sqrt{\mu_{0,m}}$ and ~$\eta(k)=kD_0(k)$. Then
\begin{align}\label{inf}
\inf_{m\in \mathbb{N}} \left|\frac{ d {\eta}(k_m)}{d k}\right|>0.
\end{align}
If~$|\xi|=|a-b|$, obviously we can obtain \eqref{inf}. If~$|\xi|<|a-b|$, we can prove \eqref{inf} by using the method in ~\cite[pp. 548-549]{hald}.
In fact,
\begin{align}
\left|\frac{ d {\eta}(k_m)}{d k}\right|=&|\alpha_1(b-a)\cos k_m (b-a)-\alpha_2 \xi\cos k_m\xi| \nonumber \\
&\ge\alpha_1|b-a|\left(\sqrt{1-\frac{\alpha_2^2}{\alpha_1^2}\sin^2 k_m\xi}-\frac{\alpha_2|\xi|}{\alpha_1|b-a|}\sqrt{1-\sin^2 k_m\xi}\right) \nonumber \\
& \equiv \alpha_1|b-a| A_1(k_m). \nonumber
\end{align}
If~$|\xi|/|b-a|\ge \alpha_2/\alpha_1$, then the minimum of~ $A_1(k_m)$ is~$1-({\alpha_2|\xi|}/({\alpha_1|b-a|}))$. If~$|\xi|/|b-a|< \alpha_2/\alpha_1$,
then the minimum of~$A_1(k_m)$ is
 ~$(1-(\alpha_2/\alpha_1)^2)^{1/2}(1-(|\xi|/|b-a|)^2)^{1/2}$. Therefore, for any ~$m\in\mathbb{N}$,
\begin{align}
\left|\frac{ d {\eta}(k_m)}{d k}\right|\ge \alpha_1|b-a|\left(1-\frac{\alpha_2^2}{\alpha_1^2}\right)^{1/2}\left(1-\frac{\xi^2}{|b-a|^2}\right)^{1/2}.
\end{align}

If~$|\xi|> |a-b|$, ~$\mu_{0,m}$ is not necessarily a simple zero of~$D_0(\lambda)$. For example, assume that~$$D_0(\lambda)=\frac{2\sin \sqrt{\lambda}-\sin 2\sqrt{\lambda}} {\sqrt{\lambda}}.$$
Then~$\mu_{0,m}=m^2\pi^2,m=1,2,\cdots$.  For any even ~$m$, one has
\begin{align}\nonumber
\frac{ d {\eta}(k_m)}{d k}=0.
\end{align}

We have the following lemma.
 \begin{lemma}\label{mum}
 	Assume that ~$\sigma\in L^2(0,a)$ and $a\ne b$. If $\rho(b)=1$ and ~$|\xi|\le|a-b|$, then problem ~\eqref{eq215} has real eigenvalues ~$\{\mu_m\}_{m=m_0+1}^\infty$ satisfying
 	\begin{align}\label{pertubation}
 	\sqrt{\mu_m}=\sqrt{\mu_{0,m}}+\kappa_m,
 	\end{align}
 	where~$\{\kappa_m\}_{m=m_0+1}^\infty\in \ell^2$.  \end{lemma}
 \begin{proof}
 	By Lemma \ref{lemma22}, \eqref{eq1.4=} and \eqref{linearproperty}, the characteristic function $D(\lambda)$ has the representation
 	\begin{align}\label{dlambdainte}
 	D(\lambda)= \frac{1}{\rho(0)^{1/4}}\left(\alpha_1\frac{\sin( \sqrt{\lambda} (b-a)) }{\sqrt{\lambda}}
 	-\alpha_2\frac{\sin (\sqrt{\lambda}\xi )}{\sqrt{\lambda}}+\int_{b-a}^{b+a} h(t)
 	\frac{\sin(\sqrt{\lambda}t)}{\sqrt{\lambda}}dt\right).
 	\end{align}
 	Here $h\in L^2(b-a,a+b)$.
 	According to \eqref{dlambdainte} and Riemann-Lebesgue lemma, there exists ~$m_0$,  so that for~$m>m_0$, ~$D(x_{1,m})D(x_{2,m})<0$.
 	Therefore for ~$m>m_0$, ~$D(\lambda)$  has at least one zero  $\mu_{m}$ on  $(x_{1,m}, x_{2,m})$. Namely,
 	\begin{align}
 	\mu_{m}\in (x_{1,m}, x_{2,m}). \label{large1}
 	\end{align}
 	
 	We next show that $\left\{\int_{b-a}^{b+a} h(t) \sin \sqrt{\mu}_{m} t dt\right\}_{m=m_0+1}^\infty \in \ell^2$. Without loss of generality, assume that $a>b$. We first prove that for any $f\in L^2(0,a-b)$, there holds
 	$$\left\{\int_0^{a-b} f(t) \sin \sqrt{\mu}_{m} t dt\right\}_{m=m_0+1}^\infty \in \ell^2.$$
 	Because ~$\left\{\sqrt{\frac{{2}}{a-b}}\sin (n\pi t/(a-b))\right\}_{n=1}^\infty$ is an
 	orthonormal basis in~$L^2(0,a-b)$, then there exists ~$\{\beta_n\}_{n=1}^\infty \in \ell^2$, so that
 	\begin{align}\nonumber
 	f(t)=\sum_{n=1}^\infty \beta_n \sin \frac{n\pi t}{a-b}.
 	\end{align}
 	Therefore
 	\begin{align}
 	& \int_0^{a-b} f(t) \sin \sqrt{\mu}_{m}t dt=\sum_{n=1}^\infty \beta_n \int_0^{a-b} \sin\sqrt{\mu}_{m}t \sin \frac{n\pi t}{a-b} tdt \nonumber \\
 	&=\sum_{n=1}^\infty \frac{\beta_n}{2} \left(-\int_0^{a-b}\cos\left(\sqrt{\mu}_m+\frac{n\pi}{a-b}\right)t+\cos\left(\sqrt{\mu}_m-\frac{n\pi}{a-b}\right)tdt\right) \nonumber \\
 	&=\sum_{n=1}^\infty \frac{O(\beta_n)}{\sqrt{\mu}_m+\frac{n\pi}{a-b}}
 	+\sum_{n=1, n\ne m}^\infty \frac{O(\beta_n)}{\sqrt{\mu}_m-\frac{n\pi}{a-b}}+O(\beta_m). \label{muinl2}
 	\end{align}
 	From ~\eqref{large1} and the fact that ~$\{\beta_n\}_{n=1}^\infty \in \ell^2$, one has
 	\begin{align}\nonumber
 	\sum_{m=m_0+1}^\infty\sum_{n=1}^\infty \frac{|\beta_n|^2}{(\sqrt{\mu}_m
 		+\frac{n\pi}{a-b})^2}\le\sum_{m=m_0+1}^\infty\sum_{n=1}^\infty \frac{|\beta_n|^2}{\mu_m}<\infty.
 	\end{align}
 Then
 	\begin{align}\label{plus}
 	\begin{Bmatrix}
 	\sum_{n=1}^\infty \frac{\beta_n}{\sqrt{\mu}_m+\frac{n\pi}{a-b}}
 	\end{Bmatrix}_{m=m_0+1}^\infty
 	\in \ell^2.
 	\end{align}
 	According to~\eqref{fuliye},
 	\begin{align}\label{minus}
 	\begin{Bmatrix}
 	\sum_{n=1,n\ne m}^\infty \frac{\beta_n}{\sqrt{\mu}_m-\frac{n\pi}{a-b}}
 	\end{Bmatrix}_{m=m_0+1}^\infty
 	\in \ell^2.
 	\end{align}
 	By~\eqref{muinl2}, \eqref{plus} and~\eqref{minus}, we know $$\left\{\int_0^{a-b} f(t) \sin \sqrt{\mu}_{m} t dt\right\}_{m_0+1}^\infty \in \ell^2.$$
 	Let~$\left\lfloor{2a}/(a-b)\right\rfloor$ be the largest integer not exceeding ~${2a}/(a-b)$. Define
 	\begin{align}\label{st}
 	h_1(t)=
 	\begin{cases}
 	h(t),  &  b-a<t<b+a, \\
 	0, &   b+a<t<(a-b)\left\lfloor\frac{2a}{a-b}+1\right\rfloor.
 	\end{cases}
 	\end{align}
 	Using the periodic properties of trigonometric functions, for any~$j=0,\cdots, \left\lfloor{2a}/(a-b)\right\rfloor$,  we have
 	\begin{align}\nonumber
 	\begin{Bmatrix}
 	\int_{b-a+j(a-b)}^{b-a+(j+1)(a-b)} h_1(t)\sin \sqrt{\mu}_{m} t dt
 	\end{Bmatrix}_{m=m_0+1}^\infty
 	\in \ell^2.
 	\end{align}
 Hence
 	\begin{align}\nonumber
 	\begin{Bmatrix}
 	\int_{b-a}^{(a-b)\left\lfloor\frac{2a}{a-b}+1\right\rfloor} h_1(t) \sin \sqrt{\mu}_{m} t dt
 	\end{Bmatrix}_{m=m_0+1}^\infty
 	\in \ell^2.
 	\end{align}
 	By~\eqref{st}, one concludes that ~$$\left\{\int_{b-a}^{b+a} h(t) \sin \sqrt{\mu}_{m} t dt\right\}_{m=m_0+1}^\infty \in \ell^2.$$
 	
 	We next show that \eqref{pertubation} holds.
 	Substituting ~$\lambda=\mu_m$ into~\eqref{dlambdainte}, using~\eqref{pertubation} and the Taylor expansion of trigonometric functions, there exists~$\{\beta_m\}_{m=m_0+1}^\infty\in \ell^2$, so that
 	\begin{align}\nonumber
 	(\alpha_1(b-a)\cos\sqrt{\mu_{0,m}} (b-a)-\alpha_2\xi\cos\sqrt{\mu_{0,m}}\xi)\kappa_m+O(\kappa_m^2)=\beta_m.
 	\end{align}
 	From~\eqref{inf}, we conclude that $\{\kappa_m\}_{m=m_0+1}^\infty\in \ell^2$. The proof is completed.
 \end{proof}

\begin{remark}
Notice that if $|\xi|\le |a-b|$, then for $m$ large enough, $D_0(\lambda)$ has exactly one zero $\mu_{0,m}$ on the interval $((m-\frac{1}{2})^2\pi^2/(a-b)^2,(m+\frac{1}{2})^2\pi^2/(a-b)^2).$ By Lemma \ref{mum}, we conclude that $D(\lambda)$ also has exactly one zero $\mu_{m}$ on  $((m-\frac{1}{2})^2\pi^2/(a-b)^2,(m+\frac{1}{2})^2\pi^2/(a-b)^2).$

\end{remark}

 \begin{lemma}\label{zero11}
 	Assume that~$|a-b|\ge|\xi|$ and~$\{\mu_{0,m}\}_{m=1}^\infty$ is the zero of~$D_0(\lambda)$. Assume that the positive sequence ~$\{\mu_m\}_{m=m_0+1}^\infty$ satisfies
 	\begin{align}
 	\sqrt{\mu_m}=\sqrt{\mu_{0,m}}+\kappa_m, \label{mui}
 	\end{align}
 	where~$\{\kappa_m\}_{m=m_0+1}^\infty\in \ell^2$. Then the infinite product
 	\begin{align}
 	g(\lambda)=\prod_{m=m_0+1}^\infty \left(1-\frac{\lambda}{\mu_m}\right)  \label{g}
 	\end{align}
 	is an entire function with respect to $\lambda$.
 	Moreover, there exist ~$C>c>0$, so that in the sector ~$\Lambda_{\delta}=\{\lambda\in \mathbb{C}| \delta<\arg(\lambda)<\pi-\delta, \delta\in (0,\pi/2)\}$,
 	\begin{align}
 	c|\sin (\sqrt{\lambda}(a-b))||\lambda|^{-m_0-1/2}<|g(\lambda)|<C|\sin( \sqrt{\lambda}(a-b))||\lambda|^{-m_0-1/2} \label{infinite}
 	\end{align}
 \end{lemma}

 \begin{proof}
 By ~\eqref{large1}, the series~$\sum_{m>m_0} |\lambda/\mu_m|$ converges uniformly on the bounded set of ~$\lambda$ plane. Therefore, the infinite product~\eqref{g} converges uniformly on  the bounded set of  ~$\lambda$ plane
 and hence~$g(\lambda)$ is an entire function.

 Notice that
 \begin{align}\nonumber
 D_0(\lambda)=(\alpha_1(b-a)-\alpha_2\xi)\prod_{m=1}^\infty \left(1-\frac{\lambda}{\mu_{0,m}}\right).
 \end{align}
 Therefore
 \begin{align}
 \frac{g(\lambda)}{D_0(\lambda)}=(\alpha_1(b-a)-\alpha_2\xi)\prod_{m=1}^{m_0}\frac{\mu_{0,m}}{\mu_{0,m}-\lambda}
 \times\prod_{m=m_0+1}^\infty \frac{\mu_{0,m}(\mu_m-\lambda)}{\mu_{m}(\mu_{0,m}-\lambda)} \label{c0}.
 \end{align}
 For ~$\lambda\in\Lambda_\delta$, there exist ~$C_1>c_1>0$, such that
 \begin{align} c_1|\lambda|^{-m_0}<\prod_{m=1}^{m_0}\left|\frac{\mu_{0,m}}{\mu_{0,m}-\lambda}\right|<C_1|\lambda|^{-m_0}. \label{c1}
 \end{align}
 From ~\eqref{large1} and ~\eqref{mui}, for ~$m>m_0$,
 \begin{align}\nonumber
 \frac{\mu_{0,m}}{\mu_{m}}=1+\frac{\beta_m}{m},
 \end{align}
 where ~$\{\beta_m\}_{m=1}^\infty\in \ell^2$. Then by ~Cauchy-Schwarz inequality, the series ~$\sum_{m>m_0} (1-\mu_{0,m}/\mu_{m})$ converges and there exist~$C_2>c_2>0$, such that
 \begin{align}
 0<c_2<\prod_{m=m_0+1}^\infty\frac{\mu_{0,m}}{\mu_{m}}<C_2. \label{c2}
 \end{align}

 For~$|\lambda|=(n+1/2)^2\pi^2/(a-b)^2$, $n=1,2,\cdots$, we have
 \begin{align}\nonumber
 \frac{\mu_m-\lambda}{\mu_{0,m}-\lambda}=
 \begin{cases}
 1+O(\kappa_n),& m=n, \\
 1+O\left(\frac{m\beta_m}{m^2-n^2}\right),& m\ne n.
 \end{cases}
 \end{align}
 Thus, applying  Lemma~\ref{pt} to
 \begin{align}\nonumber
 a_{mn} =
 \begin{cases}
 0, & m\le m_0 \mbox{ or} ~n\le m_0, \\
 \frac{\mu_m-\lambda}{\mu_{0,m}-\lambda}-1, & m, n>m_0
 \end{cases}
 \end{align}
 with~$|\lambda|=(n+1/2)^2\pi^2/(a-b)^2$, $n=1,2,\cdots$, we have
 \begin{align} \prod_{m=m_0+1}^\infty \frac{\mu_m-\lambda}{\mu_{0,m}-\lambda}=1+o(1), n\rightarrow\infty. \label{i}
 \end{align}

We next show that for any ~$\lambda\in \Lambda_\delta$, ~\eqref{i} holds.
 Note that for any fixed ~$\lambda$, there exists ~$n_0$, such that $$(n_0-1/2)^2\pi^2/(a-b)^2\le|\lambda|<(n_0+1/2)^2\pi^2/(a-b)^2.$$ Therefore,  it suffices to prove that there exists  ~$C>0$,  which is independent of ~$\lambda, m, n_0$,
 so that for any ~$m>m_0$,
 \begin{align}\label{xianran}
 \frac{1}{|\lambda-\mu_{0,m}|}\le C\frac{1}{|(n_0+1/2)^2\pi^2/(a-b)^2-\mu_{0,m}|}.
 \end{align}
 The proof of ~\eqref{xianran} is obvious and we omit the steps.

 By~\eqref{c0}, \eqref{c1}, \eqref{c2} and~ \eqref{i}, we  can obtain \eqref{infinite}. The lemma is proved.  	
 \end{proof}

When $a\ne b$,  problem $Q(\rho)$ has  the ``almost real subspectrum'' $\{\mu_m\}_{m=1}^\infty$. Recall that the set of all eigenvalues of $Q(\rho)$ is denoted by  $\{\lambda_k\}_{k=1}^\infty$.

 \begin{theorem}\label{geshu}
 	Assume that {$\rho\in W_2^1\left((0, b_1)\cup (b_1, b)\right)$ and satisfies} \eqref{piecewiseac} and $\rho(b)= 1$. Assume that   $a\ne b$.  Then $Q(\rho)$ has a subsequence of  eigenvalues $\{\mu_m\}_{m=1}^\infty$
  satisfying
 	\par{\rm{(i)}} there exists $m_0\in \mathbb{N}$ such that for $m=1,\cdots, m_0$, we have $|\mu_m|<(m_0+\frac{1}{2})^2\pi^2/(a-b)^2$;
 	\par{\rm{(ii)}} for $m> m_0$, all $\mu_m$ are real and satisfy $$(m-\frac{1}{2})^2\pi^2/(a-b)^2<|\mu_{m}|<(m+\frac{1}{2})^2\pi^2/(a-b)^2.$$ \\
 	Moreover, if
 	\par{\rm{(1)}} $0< b<a $ and $\int_{b_1}^b \sqrt{\rho(r)}dr \ge b$, \\
 	or
 	\par{\rm{(2)}} $b>a$,\\
 	then $\{\mu_m\}_{m=1}^\infty$ is a proper set of $\{\lambda_k\}_{k=1}^\infty$.
 \end{theorem}
 \begin{proof}
 By \eqref{large1}, we know that  (ii) holds. 	Let $N(r)$ be the number of zeros of $D(\lambda)$ in the circle $|\lambda|<r$.
  By Lemma \ref{lemma22} and Lemma \ref{zero11}, there exists $n_0$, for $n>n_0$, we can obtain

 	\begin{align}\label{number}
 	N\left(\frac{(n+1/2)^2\pi^2}{(a-b)^2}\right)   \ge n.
 	\end{align}
 	 The proof of \eqref{number} is similar to that of \cite[Lemma 5]{mcl} and we omit the steps.
 	By \eqref{number}, we conclude that (i) holds.
 	
 	Moreover, if (1) or (2) of Theorem \ref{geshu} is satisfied, we have ~$|\xi|> |a-b|$.  Then $D(\lambda)$ is an entire function of order ~1/2 and type greater than  ~$|a-b|$.
  From ~\cite[p. 127]{lev},  $D(\lambda)$ has infinite zeros besides ~$\{\mu_m\}_{m=1}^\infty$. Therefore $\{\mu_m\}_{m=1}^\infty$ is a proper set of $\{\lambda_k\}_{k=1}^\infty$. The proof is completed.
 	\end{proof}

Now in a position to state and prove our main result in this section. The following theorem considers the mixed spectral problem \cite{mak}. That is,  recover  the refractive index from the ``almost real subspectrum'' $\{\mu_m\}_{m=1}^\infty$ and partial information on the refractive index.
Theorem~\ref{theorem1} refines the refractive index in \cite{mcl} from a~$W_2^2$ function to a piecewise~$W_2^1$ function.  Moreover, we drop the condition $\rho'(b)=0$
in
\cite{mcl}.

\begin{theorem}\label{theorem1}
	Assume that {$\rho\in W_2^1\left((0, b_1)\cup (b_1, b)\right)$ and satisfies} \eqref{piecewiseac} and $\rho(b)= 1$.	Suppose that $a$ is known. Assume that one of the following four conditions holds:
	\par	{\rm{(1)}} $0< b<a $ and ~$\rho$ is known on the interval ~$[A,b]$, where ~$A$ satisfies ~$\int_A^b \sqrt{\rho(r)} dr=(a+b)/2$, $\int_{b_1}^b \sqrt{\rho(r)}dr \ge b$;
	\par	{\rm{(2)}} $0< b<a $ and ~$\rho$ is known on the interval ~$[A,b]$, where ~$A$ satisfies ~$\int_A^b \sqrt{\rho(r)} dr=(a+b)/2$, $\int_{b_1}^b \sqrt{\rho(r)}dr < b$, ~ one of the eigenvalues, denoted by $\mu_0$, in $\{\lambda_k\}_{k=1}^\infty\setminus \{\mu_m\}_{m=1}^\infty$ is known;
	\par {\rm{(3)}} $a<b\le 3a$ and ~$\rho$ is known  on the interval ~$[A,b]$, where ~$A$ satisfies ~$\int_A^b \sqrt{\rho(r)} dr=(3a-b)/2$, one of the eigenvalues, denoted by $\mu_0$, in $\{\lambda_k\}_{k=1}^\infty\setminus \{\mu_m\}_{m=1}^\infty$ is known;
	\par	{\rm{(4)}} $3a<b$. \\
	Then ~$\rho$ is uniquely determined by $\{\mu_m\}_{m=1}^\infty$.
\end{theorem}

 \begin{proof}
We require that if a certain symbol $\gamma$ denotes an object related to $Q(\rho)$,
then the corresponding symbol $\tilde{\gamma}$ denotes the analogous object related to $Q(\tilde{\rho})$.

Define
 \begin{align}\nonumber
 G(\lambda)=
 \begin{cases}
 \prod_{m=1}^\infty \left(1-\frac{\lambda}{\mu_m}\right), & \mbox{ case ~(1) or ~(4)}, \\
 \prod_{m=0}^\infty \left(1-\frac{\lambda}{\mu_m}\right), & \mbox{ case ~(2) or ~(3)}.
 \end{cases}
 \end{align}
 In case ~(1), $|a-b|\ge |\xi|$, by Lemma ~\ref{zero11}, in the sector ~$\Lambda_{\delta}$,
 \begin{align}\label{glambda11}
 |G(\lambda)|>c|\sin (\sqrt{\lambda}(a-b))||\lambda|^{-1/2}.
 \end{align}
 In case (2) or ~(3), letting ~$\lambda=iy,y\in\mathbb{R}$, then there exists ~$C_0>0$, such that
 \begin{align}\label{glambda1}
 |G(iy)|\ge C_0\prod_{m=1}^\infty \left|1+\frac{y^2}{(m-1/2)^4\pi^4/(a-b)^4}\right|^{1/2}
 =C_0\left|{\cos(\sqrt{iy}(a-b))}\right|.
 \end{align}
 In case ~(4), letting ~$\lambda=iy,y\in\mathbb{R}$, then there exists ~$C_0>0$, such that
 \begin{align}\label{glambda111111111}
 |G(iy)|\ge C_0\left|\frac{{\cos(\sqrt{iy}(a-b))}}{y}\right|.
 \end{align}

 	Define
 	\begin{align}
 	H(\lambda)= z(a,\lambda)\tilde{z}^{[1]}(a,\lambda)-\tilde{z}(a,\lambda)z^{[1]}(a,\lambda), \label{definitionhz1}
 	\end{align}
 	Then in ~$\Sigma_{\delta}:=\Lambda_{\delta}\cup\{\lambda\in \mathbb{C}| \pi+\delta<\arg(\lambda)<2\pi-\delta, \delta\in (0,\pi/2)\}$, we have
 	\begin{align}\label{1234}
 	H(\lambda)=
 	\begin{cases} o\left(\frac{e^{|a-b||\rm{Im}\sqrt{\lambda}|}}{\sqrt{\lambda}}\right), & \mbox{ case ~(1), (2) or ~(3)}, \\
 	o\left(\frac{e^{2a|\rm{Im}\sqrt{\lambda}|}}{\sqrt{\lambda}}\right), & \mbox{ case ~(4)}.
 	\end{cases}
 	\end{align}
 	We only prove  \eqref{1234} in case ~(1), and other cases can be proved similarly. In case (1),
 	\begin{align}
 	s\left(a,\lambda; \left(\frac{a-b}{2}\right)-\right)&=\tilde{s}\left(a, \lambda; \left(\frac{a-b}{2}\right)-\right),\nonumber \\
 	 s^{[1]}\left(a, \lambda; \left(\frac{a-b}{2}\right)-\right)&=\tilde{s}^{[1]}\left(a, \lambda; \left(\frac{a-b}{2}\right)-\right),\nonumber \\
 	c\left(a, \lambda; \left(\frac{a-b}{2}\right)-\right)&=\tilde{c}\left(a, \lambda; \left(\frac{a-b}{2}\right)-\right),\nonumber \\
 	 c^{[1]}\left(a,\lambda; \left(\frac{a-b}{2}\right)-\right)&=\tilde{c}^{[1]}\left(a, \lambda; \left(\frac{a-b}{2}\right)-\right). \nonumber
 	\end{align}
 	Here ~$s(x,\lambda;y), c(x,\lambda;y)$ are normalized according to~\eqref{normalized}.  Letting ~$x=a, y=((a-b)/2)-$ in ~\eqref{linearcom} and ~\eqref{linearcom1},  substituting  them into ~\eqref{definitionhz1},
 	then using  ~\eqref{langsiji} and ~\eqref{linearproperty}, one knows
 	\begin{align}
 	H(\lambda)=&z\left(\left(\frac{a-b}{2}\right)-,\lambda\right)\tilde{z}^{[1]}\left(\left(\frac{a-b}{2}\right)-,\lambda\right)-\tilde{z}\left(\left(\frac{a-b}{2}\right)-,\lambda\right)z^{[1]}\left(\left(\frac{a-b}{2}\right)-,\lambda\right)\nonumber \\
 	=&z\left(\left(\frac{a-b}{2}\right)-,\lambda\right)\tilde{z}\left(\left(\frac{a-b}{2}\right)-,\lambda\right)\left(m\left(\left(\frac{a-b}{2}\right)-,\lambda\right)-\tilde{m}\left(\left(\frac{a-b}{2}\right)-,\lambda\right)\right). \nonumber
 	\end{align}
 	 From Lemma~\ref{lemma22} and~\eqref{aaaaa}, in the sector ~$\Lambda_{\delta}$,  $H(\lambda)$ has asymptotic form ~\eqref{1234}.
 	Since ~$H(\lambda)$ is a real entire function, then ~$H(\lambda)$ also has asymptotic form ~\eqref{1234} in ~$\Sigma_{\delta}$.

 	We next show that ~$H(\lambda)\equiv 0$. Define
 	\begin{align}\nonumber
 	F(\lambda):=\frac{H(\lambda)}{G(\lambda)}.
 	\end{align}
 	Arguing as in Theorem ~\ref{a=b}, we know  $F(\lambda)$ is an entire function. In case ~(1), by ~\eqref{glambda11} and~\eqref{1234},
 as ~$|\lambda|\rightarrow \infty$	in the sector ~$\Lambda_\delta$,
  we have ~$F(\lambda)=o(1)$.
 	In cases ~(2), (3) or ~(4), by  \eqref{glambda1}, ~\eqref{glambda111111111} and \eqref{1234}, as ~$|\lambda|\rightarrow \infty$ on the imaginary axis, one has ~$F(\lambda)=o(1)$.
 	Note that in the above four cases, using ~Phragm\'{e}n-Lindel\"{o}f theorem and ~Liouville theorem, we can get ~$F(\lambda)\equiv0$.
 	From ~\eqref{definitionhz1}, one can obtain that $m(a,\lambda)\equiv \tilde{m}(a,\lambda)$.
 	Using Theorem ~\ref{uniqueness} and Lemma ~\ref{unique}, we have ~$\rho\equiv \tilde{\rho}.$
 \end{proof}

\vspace{2mm}\noindent \textbf{Data availability}
No data was used for the research described in the article.

\vspace{2mm}\noindent \textbf{Conflict of interest} The authors report no conflict of interest.

\end{document}